\documentclass[11pt]{amsart}
\usepackage{amssymb}

\newtheorem{theorem}{Theorem}[section]
\newtheorem{corollary}[theorem]{Corollary}
\newtheorem{lemma}[theorem]{Lemma}
\newtheorem{conjecture}[theorem]{Conjecture}
\newtheorem{proposition}[theorem]{Proposition}
\theoremstyle{definition}

\newtheorem{question}[theorem]{Question}
\theoremstyle{remark}

\numberwithin{equation}{section}

\def\N{{\mathbb{N}}}

\def\Z {{\mathbb{Z}}}

\def\supp{{\mathrm{supp}\,}}

\def\3{{|\!|\!|}}

\def\lh{{\mathrm{lh}}}
\def\het{{\mathrm{ht}}}
\def\acts{{\curvearrowright}}

\begin{document}

\title{Non-Archimedean Abelian Polish Groups and Their Actions}

\author{Longyun Ding}
\address{School of Mathematical Sciences and LPMC \\ Nankai University \\ Tianjin 300071 \\ P.R. China}
\email{dinglongyun@gmail.com}
\author{Su Gao}
\address{Department of Mathematics\\ University of North Texas\\ 1155 Union Circle \#311430\\  Denton, TX 76203\\ U.S.A.}
\email{sgao@unt.edu}

\date{\today}
\thanks{The first author's research is partially supported by the National Natural Science Foundation of China (Grant No. 11371203). The second author's research is partially supported by the US NSF grant DMS-1201290. Both authors acknowledge the Institute for Mathematical Sciences at the National University of Singapore for sponsoring a research visit during which this research was done.}
\maketitle \thispagestyle{empty}

\begin{abstract} 
In this paper we consider non-archimedean abelian Polish groups whose orbit equivalence relations are all Borel. Such groups are called tame. We show that a non-archimedean abelian Polish group is tame if and only if it does not involve $\Z^\omega$ or $(\Z(p)^{<\omega})^\omega$ for any prime $p$. In addition to determining the structure of tame groups, we also consider the actions of such groups and study the complexity of their orbit equivalence relations in the Borel reducibility hierarchy. It is shown that if  such an orbit equivalence relation is essentially countable, then it must be essentially hyperfinite. We also find an upper bound in the Borel reducibility hierarchy for the orbit equivalence relations of all tame non-archimedean abelian Polish groups. 
\end{abstract}


\section{Introduction}
A Polish group $G$ is {\em tame} if any orbit equivalence relation induced by a Borel action of $G$ on a Polish space is Borel. It is well known that locally compact Polish groups are tame. Solecki \cite{solecki} showed that the converse is not true. In fact he completely characterized tameness among product groups of the form $\prod_{n\in\omega} H_n$, where each $H_n$ is countable discrete abelian. It follows from his results that $\mathbb{Z}^\omega$ and $(\mathbb{Z}(p)^{<\omega})^\omega$, for any prime $p$, are examples of Polish groups that are not tame. If $G$ and $H$ are Polish groups, we say that $G$ {\em involves} $H$ if there is a continuous surjective homomorphism from a closed subgroup of $G$ onto $H$. If $G$ involves $H$ and $G$ is tame, then so is $H$.

In this paper we completely characterize tameness for all non-archimedean abelian Polish groups. These are exactly all closed subgroups of product groups of the form $\prod_n H_n$, where each $H_n$ is countable discrete abelian. Our result thus generalizes Solecki's. The following is our first main theorem.

\begin{theorem} \label{main1} Let $G$ be a non-archimedean abelian Polish group. Then the following are equivalent:
\begin{enumerate}
\item[\rm (i)] $G$ is tame, i.e., any orbit equivalence relation induced by a Borel action of $G$ on a Polish space is Borel;
\item[\rm (ii)] There is a decreasing sequence $(G_k)_{k\in\omega}$ of open subgroups of $G$ which is a nbhd base of the identity element of $G$ such that 
\begin{itemize}
\item[\rm (a)] for all but finitely many $k$, $G_k/G_{k+1}$ is torsion, and, 
\item[\rm (b)] for any prime $p$ and for all but finitely many $k$, $G_k/G_{k+1}$ contains only finitely many elements of order $p$;
\end{itemize}
\item[\rm (iii)] $G$ does not involve either $\mathbb{Z}^\omega$ or $(\mathbb{Z}(p)^{<\omega})^\omega$ for any prime $p$.
\end{enumerate}
\end{theorem}

Clause (ii) can be viewed as a description of the structure of all tame non-archimedean abelian Polish groups. A similar structural theorem for the so-called quasi-reduced case was independently obtained by Malicki \cite{Ma2} recently. Clause (iii) says that $\mathbb{Z}^\omega$ and $(\mathbb{Z}(p)^{<\omega})^\omega$ are canonical examples of non-tame non-archimedean abelian Polish groups, that is, any non-tame such groups must involve one of these canonical examples. 

We also consider the orbit equivalence relations induced by Borel actions of tame non-archimedean abelian Polish groups. Our main results give some upper bounds for these equivalence relations in the Borel reducibility hierarchy. For equivalence relations $E, F$ on Polish spaces $X, Y$, respectively, we say that $E$ is {\em Borel reducible} to $F$, denoted $E\leq_B F$, if there is a Borel function $\varphi: X\to Y$ such that for all $x, x'\in X$, we have $xEx'\iff \varphi(x)F\varphi(x')$. We will refer to the equivalence relation $E_0$ on $2^\omega$, which is defined by 
$$ xE_0 y\iff \exists n\ \forall m\geq n\ x(m)=y(m). $$
If $E$ is an equivalence relation on a Polish space $X$, then we also define the following equivalence relations $E^\omega$ and $E^+$ on $X^\omega$:
$$ (x_n)E^\omega (y_n)\iff \forall n\ x_nEy_n $$
$$ (x_n)E^+(y_n)\iff \forall n\ \exists m\ x_n Ey_m\mbox{ and } \forall m\ \exists n\ x_n Ey_m. $$
Our main results on the complexity of the orbit equivalence relations are as follows.

\begin{theorem}\label{main3} If $G$ is a tame non-archimedean abelian Polish group, then any orbit equivalence relation induced by a Borel action of $G$ is Borel reducible to $(E_0^\omega)^{+++}$.
\end{theorem}

\begin{theorem}\label{main2} If $G$ is a locally compact non-archimedean abelian Polish group, then any orbit equivalence relation induced by a Borel action of $G$ is Borel reducible to $E_0$.
\end{theorem}

Both proofs use a theorem of Gao and Jackson \cite{GJ} on the hyperfiniteness of orbit equivalence relations of countable abelian group actions. Theorem~\ref{main2} can be viewed as a generalization of the Gao--Jackson theorem, since countable discrete groups are locally compact. Theorem~\ref{main2} is in fact deduced from a more general result as follows.

\begin{theorem}\label{main4} Let $G$ be a non-archimedean abelian Polish group and let $F$ be an orbit equivalence relation induced by a Borel action of $G$ on a Polish space. If $E\leq_B F$ and $E$ is essentially countable, then $E$ is essentially hyperfinite (i.e. $E\leq_B E_0$).
\end{theorem}

The rest of the paper is organized as follows. In Section 2 we define some basic concepts, fix some basic notation, and recall some basic results of the subject. In Section 3 we give a general development of group trees and coset trees. These are slightly more general than Solecki's results but are mostly straightforward generalizations. We will need these results in the subsequent sections. In Section 4 we prove Theorem~\ref{main1} and give some corollaries. In Section 5 we compute some upper bounds for the ranks of coset trees and group trees on tame non-archimedean abelian Polish groups, and use these results to prove Theorem~\ref{main3}. In Section 6 we give the (short) proof of Theorem~\ref{main4} and deduce some consequences of the theorem. In Section 7 we construct a universal group among all tame product groups. Finally in Section 8 we give some additional remarks and mention some further open problems.

\section{Background}

\subsection{Non-archimedean Polish groups}

A topological group is {\em Polish} if it is separable and completely metrizable. A Polish group is {\em non-archimedean} if it has a nbhd base of its identity element consisting of open subgroups. A prototypical example of non-archimedean Polish group is the {\em infinite permutation group} $S_\infty$, i.e., the group of all permutations of $\mathbb{N}$ with the pointwise convergence topology. By a theorem of Becker and Kechris (\cite{BK}; also c.f. \cite{GaoBook} Theorem 2.4.1), a Polish group is non-archimedean if and only if it is isomorphic to a closed subgroup of $S_\infty$.

For non-archimedean abelian Polish groups, the following are some useful characterizations.

\begin{proposition}\label{prop:abel} Let $G$ be a Polish group. Then the following are equivalent:
\begin{enumerate}
\item[\rm (i)] $G$ is non-archimedean abelian;
\item[\rm (ii)] $G$ is isomorphic to a closed abelian subgroup of $S_\infty$;
\item[\rm (iii)] $G$ is isomorphic to a closed subgroup of a product group $\prod_{n\in\omega}H_n$, where each $H_n$ is countable discrete abelian;
\item[\rm (iv)] $G$ is pro-countable abelian, i.e., there is an inverse system
$$ \Gamma_0\leftarrow \Gamma_1\leftarrow\cdots \Gamma_n\leftarrow \cdots $$
of countable discrete abelian groups, with projection maps $\pi_{i,j}: \Gamma_i\to \Gamma_j$ for $i>j$, such that $G$ is isomorphic to the inverse limit
$$ \varprojlim_{n}\Gamma_n=\left\{(\gamma_n)\in\prod_n \Gamma_n\,:\, \forall n\ \pi_{n+1,n}(\gamma_{n+1})=\gamma_n\right\}. $$
\end{enumerate}
\end{proposition}

Occasionally we will also consider a slight generalization of conditions (iii) and (iv) in which the abelian assumption is removed. It is perhaps worth noting that this case corresponds exactly to non-archimedean Polish groups which admit two-sided invariant metrics. A Polish group $G$ is {\em TSI} if it admits a compatible metric $d$ which is two-sided invariant, i.e., for all $g, h, k\in G$, $d(gh, gk)=d(h,k)=d(hg, hk)$. 

\begin{proposition} Let $G$ be a Polish group. Then the following are equivalent:
\begin{enumerate}
\item[\rm (i)] $G$ is non-archimedean TSI;
\item[\rm (ii)] $G$ is isomorphic to a closed TSI subgroup of $S_\infty$;
\item[\rm (iii)] $G$ is isomorphic to a closed subgroup of a product group $\prod_{n\in\omega}H_n$, where each $H_n$ is countable discrete;
\item[\rm (iv)] $G$ is pro-countable.
\end{enumerate}
\end{proposition}

These results have appeared several times in the literature (\cite{GX} \cite{HK0} \cite{malicki} \cite{MS}) and thus we consider them folklore. In this paper we only need the equivalent formulations (iii) in the above propositions.

\subsection{Universality of groups and orbit equivalence relations}
If $G$ is a Polish group, $X$ is a Polish space, and $\cdot: G\acts X$ is a Borel action, we define the {\em $G$-orbit equivalence relation} $E^X_G$ on $X$ by $xE^X_Gx'\iff \exists g\in G\ g\cdot x=x'$. By a theorem of Becker and Kechris (\cite{BK}; also c.f. \cite{GaoBook} Theorem 3.3.4), for any Polish group $G$ there exists a $G$-orbit equivalence relation $E^X_G$ such that for any other $G$-orbit equivalence relation $E^Y_G$, we have $E^Y_G\leq_B E^X_G$. Such an orbit equivalence relation is called a {\em universal} $G$-orbit equivalence relation. Clearly a Polish group $G$ is tame if and only if a universal $G$-orbit equivalence relation is Borel. 

If $G$ is a Polish group and $H$ is a closed subgroup of $G$, then by a theorem of Hjorth and Mackey (c.f. \cite{GaoBook} Theorem 3.5.2) a universal $H$-orbit equivalence relation is Borel reducible to a universal $G$-orbit equivalence relation. Similarly if $G$ is a Polish group and $H$ is a topological quotient group, that is, $H$ is a continuous homomorphic image of $G$. Combining these, we get that if $G$ involves $H$, i.e., if there is a continuous surjective homomorphism from a closed subgroup of $G$ onto $H$, then a universal $H$-orbit equivalence relation is Borel reducible to a universal $G$-orbit equivalence relation. Thus if $G$ involves $H$ and $G$ is tame, then so is $H$.

Let $\mathcal{C}$ be a class of Polish groups closed under isomorphism. We say that a Polish group $G$ is {\em universal} in $\mathcal{C}$ if $G\in \mathcal{C}$ and any $H\in\mathcal{C}$ is isomorphic to a closed subgroup of $G$. Similarly, $G$ is {\em surjectively universal} in $\mathcal{C}$ if $G\in \mathcal{C}$ and any $H\in\mathcal{C}$ is isomorphic to a topological quotient group of $G$. Finally, $G$ is {\em weakly universal} in $\mathcal{C}$ if $G\in\mathcal{C}$ and any $H\in\mathcal{C}$ is involved in $G$. 

For instance, the class of all countable discrete abelian groups has universal elements as well as surjectively universal elements. To define them we fix the following notation for this paper.

Let $\omega$ and $\mathbb{N}$ denote the set of all natural numbers. We use them interchangeably, but tend to use $\omega$ as an index set and use $\mathbb{N}$ when some number theoretic properties of natural numbers are needed. In particular, we will use $\mathbb{N}_+$ to denote the set of all positive integers, and $\mathbb{P}\subseteq\mathbb{N}_+$ to denote the set of all primes. 

For a countable abelian group $A$ let $A^{<\omega}$ denote the direct sum of countably infinitely many copies of $A$, i.e., $\bigoplus_\omega A$. Let $\mathbb{Z}$ denote the additive group of all integers, $\mathbb{Q}$ denote the additive group of all rational numbers, and for each prime $p$, let $\mathbb{Z}(p^\infty)$ denote the {\em $p$-quasicyclic group}, i.e., $\mathbb{Z}(p^\infty)$ is the additive mod 1 group of
$$\left\{\frac{m}{p^k}\,:\, k\in\mathbb{N}_+,\ m\in \mathbb{N},\ 0\leq  m< p^k\right\}. $$
Then $\mathbb{Z}^{<\omega}$ is the free abelian group with countably infinitely many generators, and therefore is a surjectively universal countable discrete abelian group. The group 
$$ A_\infty=\bigoplus_{p\in\mathbb{P}} \mathbb{Z}(p^\infty)^{<\omega}\oplus \mathbb{Q}^{<\omega} $$
is a universal countable discrete abelian group (c.f. \cite{Ro} 4.1.5 and 4.1.6).

For any Polish group $\Gamma$ let $\Gamma^\omega$ denote the infinite product group $\prod_{\omega}\Gamma$. Combining the above observations with Proposition~\ref{prop:abel} (iii), we get that $(A_\infty)^\omega$ is a universal non-archimedean abelian Polish group, and $(\mathbb{Z}^{<\omega})^\omega$ is a weakly universal non-archimedean abelian Polish group. It was shown in \cite{G} and \cite{GX} that there are surjectively universal non-archimedean abelian Polish groups.



\subsection{Countable Borel equivalence relations and hyperfiniteness}
An Borel equivalence relation $E$ on a Polish space is {\em countable} if every $E$-equivalence class is countable. An equivalence relation $E$ is {\em essentially countable} if $E\leq_B F$ for some countable Borel equivalence relation $F$. If $\Gamma$ is a countable group and $\Gamma\acts X$ a Borel action, then $E^X_\Gamma$ is a countable Borel equivalence relation. By a theorem of Feldman and Moore (c.f. \cite{GaoBook} Theorem 7.1.4), any countable Borel equivalence relation on a Polish space $X$ is the orbit equivalence relation $E^X_\Gamma$ for some countable group $\Gamma$ and Borel action $\Gamma\acts X$. A theorem of Kechris (c.f. \cite{GaoBook} Theorem 7.5.2) states that any $G$-orbit equivalence relation, where $G$ is a locally compact Polish group, is essentially countable. In particular, any locally compact Polish group is tame.

The following dichotomy theorem of Hjorth and Kechris, which is called the Seventh Dichotomy Theorem (\cite{HK} Theorem 8.1), is relevant. It states that, if $G$ is any non-archimedean TSI Polish group and $E^X_G$ is a Borel $G$-orbit equivalence relation, then exactly one of the following is true for any $E\leq_B E^X_G$: either (I) $E$ is essentially countable, or (II) $E_0^\omega\leq_B E$.

An equivalence relation $F$ is {\em finite} if every $F$-equivalence class is finite. $F$ is {\em hyperfinite} if there is an increasing sequence $(F_n)_{n\in\omega}$ of finite Borel equivalence relations, i.e., $F_n\subseteq F_{n+1}$ for all $n\in\omega$, such that $F=\bigcup_n F_n$. By a theorem of Dougherty, Jackson and Kechris (c.f. \cite{GaoBook} Theorem 7.2.3) a countable Borel equivalence relation $E$ is hyperfinite if and only if $E\leq_B E_0$. We say that an equivalence relation $E$ is {\em essentially hyperfinite} if $E\leq_B E_0$. It is a theorem of Gao and Jackson \cite{GJ} that any $\Gamma$-orbit equivalence relation, where $\Gamma$ is a countable discrete abelian group, is hyperfinite. We will use this theorem several times in the proofs below.

It is wellknown that there are countable Borel equivalence relations that are not hyperfinite (c.f., e.g., \cite{GaoBook} Theorem 7.4.10).

\section{Group trees and coset trees}
Group trees and coset trees will play a key role in the proofs of our main theorems. In this section we give a development of these concepts and their properties. Most of the results in this section are straightforward consequences or generalizations of Solecki's results in \cite{solecki}. Nevertheless, we include some details here since there are some subtle differences between our approach and that of \cite{solecki} and because we will need in the subsequent sections certain results that are not explicitly stated in \cite{solecki}.

In this section we do not assume that the groups are abelian, and therefore we will use multiplicative notation for the group operations. For the rest of this section we fix a countably infinite sequence of countable discrete groups $(H_n)$ and let $G\leq \prod_n H_n$ be a closed subgroup. Let $e_H$ be the identity element of $\prod_n H_n$. For each $m\in \omega$ let $\pi_m: \prod_n H_n\to H_m$ be the projection map. 

To facilitate a descriptive set theoretic analysis of the situation, we think of $\prod_n H_n$ as the set of branches of a tree $T_H$. For each $n\in\omega$, the $n$-th level of the tree $T_H$ is the set 
$$H^n=H_0\times\cdots\times H_{n-1}. $$
Then
$$ T_H=\bigcup_{n\in\omega}H^n.$$
When $m<n$ we continue to denote by $\pi_m$ the projection map from $H^n$ to $H_m$. If $\sigma\in T_H$, the {\em length} of $\sigma$, denoted as $\lh(\sigma)$, is the unique $n\in\omega$ such that $\sigma\in H^n$. Note that technically $H^0=\emptyset$, and therefore $\lh(\sigma)>0$ for every $\sigma\in T_H$. If $\sigma=(\pi_0(\sigma), \dots, \pi_{n-1}(\sigma))\in H^n$ and $0<m<n=\lh(\sigma)$, then define $\sigma\!\upharpoonright\! m=(\pi_0(\sigma), \dots, \pi_{m-1}(\sigma))\in H^m$. For $\sigma, \tau\in T_H$, we write $\sigma\subseteq \tau$ if $\sigma=\tau\!\upharpoonright\! \lh(\sigma)$. 

For $x\in \prod_n H_n$ and $\sigma\in T_H$, we also write $\sigma\subseteq x$ if for all $n<\lh(\sigma)$, $\pi_n(\sigma)=\pi_n(x)$. For $n\in\mathbb{N}_+$ let $x\!\upharpoonright\! n$ be the unique $\sigma\in H^n$ such that $\sigma\subseteq x$.

A subset $S\subseteq T_H$ is a {\em tree} if for any $\sigma, \tau\in T_H$, whenever $\sigma\subseteq \tau$ and $\tau\in S$, we have $\sigma\in S$. $T_H$ is itself a tree. If $S\subseteq T_H$ is a tree, then we let $[S]$ denote the set of all {\em branches} of $S$, that is, 
$$[S]=\left\{ x\in \prod_n H_n\,:\, \forall n\in\N_+\ x\!\upharpoonright\! n\in S\right\}. $$
A tree $S$ is {\em wellfounded} if $[S]=\emptyset$; it is {\em illfounded} otherwise. We have $[T_H]=\prod_n H_n$. It is easy to see that a subset of $\prod_n H_n$ is closed if and only if it is the set of branches for some tree. Now that $G$ is a closed subgroup of $\prod_n H_n$, we denote
$$T_G=\{\sigma\in T_H\,:\,\exists x\in G\ \sigma\subseteq x\}.$$
Then it is easy to check that $T_G\subseteq T_H$ is a tree and $[T_G]=G$.

For any $m\in\mathbb{N}_+$, $H^m$ is obviously a group. Let $\pi^m: \prod_n H_n\to H^m$ be the projection map. $\pi^m$ is a group homomorphism. Thus $T_G\cap H^m=\pi^m[G]$ is a subgroup of $H^m$. In general, a tree $S\subseteq T_H$ is called  a {\em group tree} if $S\cap H^m$ is a subgroup of $H^m$ for all $m\in\mathbb{N}_+$. If $S$ is a group tree, then in particular $S\cap H^m\neq\emptyset$ for all $m\in\mathbb{N}_+$, and therefore $\pi^m(e_H)\in S$ for all $m\in\mathbb{N}_+$. It follows that for any group tree $S$, $e_H\in[S]$ and therefore $[S]\neq\emptyset$. Thus group trees are always illfounded.

A tree $S\subseteq T_H$ is a {\em coset tree} if for all $m\in\mathbb{N}_+$, whenever $S\cap H^m\neq\emptyset$, we have that $S\cap H^m$ is a left coset of a subgroup of $H^m$. Note that $S\cap H^m$ is a left coset if and only if, whenever $\sigma_1, \sigma_2, \sigma_3\in S\cap H^m$, we have $\sigma_1\sigma_2^{-1}\sigma_3\in S$. When $S\cap H^m$ is a left coset of a subgroup of $H^m$, the subgroup is uniquely determined: in fact it is $\sigma^{-1}(S\cap H^m)$ for any $\sigma\in S\cap H^m$. This induces a canonical way to produce a group tree from any coset tree. For any coset tree $S\subseteq T_H$, let
$$ \Gamma(S)=\bigcup_{n\in\mathbb{N}_+} \Gamma^n $$
where $\Gamma^n=\sigma_n^{-1}(S\cap H^n)$ if $\sigma_n\in S\cap H^n\neq\emptyset$, and $\Gamma^n=\{\pi^n(e_H)\}$ if $S\cap H^n=\emptyset$. Then $\Gamma(S)$ is the canonical group tree corresponding to $S$. Moreover, if $S\subseteq T_G$, then $\Gamma(S)\subseteq T_G$. A group tree is also a coset tree. Conversely, a coset tree $S$ is a group tree exactly when $\Gamma(S)=S$. In contrast with group trees, coset trees can be wellfounded. 

A crucial concept about coset trees and group trees is that of their heights. For any tree $S\subseteq T_H$, define the {\em derivative} of $S$ to be
$$ D(S)=\{\sigma\in S\,:\, \exists \tau\in S\ (\sigma\subseteq \tau\mbox{ and } \lh(\sigma)<\lh(\tau))\}. $$
Then $D(S)\subseteq S$ is a subtree of $S$. For any ordinal $\alpha$, we defined the {\em $\alpha$-th derivative} of $S$ by transfinite induction:
$$ D^0(S)=S; \ D^{\alpha+1}(S)=D(D^\alpha(S)); \ D^{\lambda}(S)=\bigcap_{\alpha<\lambda}D^{\alpha}(S), $$
if $\lambda$ is a limit ordinal. The {\em height} of $S$, denoted $\het(S)$, is the least ordinal $\alpha$ such that $D^{\alpha+1}(S)=D^{\alpha}(S)$. Also denote $D^\infty(S)=D^{\het(S)}(S)$. If $S=\emptyset$, then technically $S$ is a tree and $\het(S)=0$. But $\het(T_H)=\het(T_G)=0$ as well. By definition, when $\alpha<\het(S)$, there exists $\sigma\in D^{\alpha}(S)\setminus D^{\alpha+1}(S)$. Since $T_H$ is countable, all trees we deal with in this paper are countable, and therefore their heights are countable ordinals. Thus for all trees $S\subseteq T_H$, $D^\infty(S)=D^{\omega_1}(S)$. If $S\subseteq T_H$ and $\sigma\in S$, the {\em rank} of $\sigma$ in $S$, denoted $r_S(\sigma)$, is the least ordinal $\alpha$ such that $\sigma\in D^{\alpha}(S)\setminus D^{\alpha+1}(S)$, if such an $\alpha$ exists, and is defined to be $\omega_1$ if $\sigma\in D^{\infty}(S)$. For simplicity we denote the rank of $\sigma$ by $r(\sigma)$ whenever the ambient tree $S$ is clear from the context.


The following two lemmas explore the relationship between ranks of group trees and coset trees.

\begin{lemma}\label{lem:c2g} Let $S\subseteq T_H$ be a coset tree. Then $\het(S)<\het(\Gamma(S))+\omega$.
\end{lemma}

\begin{proof} The proof is implicit in \cite{solecki}, in the proof of Lemma 6 (ii)$\Rightarrow$(iii).
\end{proof}

\begin{lemma}\label{lem:g2c} Let $S\subseteq T_G$ be a group tree with $\het(S)>\omega$. Then there exists a sequence $(\sigma_m)$ of elements of $T_G$ such that $(\sigma_{m+1}\!\upharpoonright\!m)\sigma_m\in S$ for all $m\in\mathbb{N}_+$ and so that $\bigcup_m \sigma_m(S\cap H^m)$ is a wellfounded coset tree with height $<\omega\cdot 2$.
\end{lemma}

\begin{proof}
This is essentially Lemma 5 of \cite{solecki} except that here we require the elements to be selected from the tree $T_G$. To prove this lemma just repeat the proof of \cite{solecki} Lemma 5 and note that all elements constructed in that proof can be taken from $T_G$. 
\end{proof}

The following proposition characterizes tameness of $G$ in terms of ranks of group trees and coset trees.

\begin{proposition}\label{prop:tamebytrees}
 The following are equivalent:
\begin{enumerate}
\item[(0)] $G$ is tame.
\item[(1)] There is $\alpha<\omega_1$ such that for any wellfounded coset tree $S\subseteq T_G$, we have $\het(S)\leq\alpha$.
\item[(2)] There is $\alpha<\omega_1$ such that for any coset tree $S\subseteq T_G$, we have $\het(S)\leq\alpha$.
\item[(3)] There is $\alpha<\omega_1$ such that for any group tree $S\subseteq T_G$, we have $\het(S)\leq\alpha$.
\end{enumerate}
\end{proposition} 

\begin{proof} The proof of (0)$\Leftrightarrow$(1) is essentially the same as the proof of Lemma 2 of \cite{solecki}. For (0)$\Rightarrow$(1) it suffices to consider the action of $G$ on the space $\mathcal{T}$ of all trees $S\subseteq T_G$. Let $E$ be the $G$-orbit equivalence relation on $\mathcal{T}$. For any trees $S, S'\subseteq T_G$, consider
$$ \Phi(S, S')=\{\sigma\in T_H\,:\, \forall 0<m\leq\lh(\sigma)\ (\sigma\!\upharpoonright\! m)(S\cap H^m)=S'\cap H^m\}. $$
Then $\Phi(S, S')\subseteq T_G$ is a coset tree. $\Phi(S, S')$ is illfounded if and only if $S$ and $S'$ are in the same $G$-orbit, i.e., $(S, S')\in E$. In general, $\het(\Phi(S, S'))$ is a ${\bf\Pi}^1_1$ norm on the $(\mathcal{T}\times\mathcal{T})\setminus E$. If $G$ is tame then $E$ is Borel, and by boundedness we get an $\alpha<\omega_1$ such that $\het(\Phi(S, S'))\leq\alpha$ whenever $\Phi(S,S')$ is wellfounded. Finally note that for any coset tree $S$, $\Phi(\Gamma(S), S)=S$. Thus if $S\subseteq T_G$ is a wellfounded coset tree, then $\het(S)\leq\alpha$.

For (1)$\Rightarrow$(0) we use a theorem of Becker and Kechris (c.f. \cite{GaoBook} Theorem 3.3.4) which states that the $G$-orbit equivalence relation on $\mathcal{T}^\omega$ is a universal $G$-orbit equivalence relation. Let $F$ be the $G$-orbit equivalence relation on $\mathcal{T}^\omega$. For two sequences of trees $(S_n), (S'_n)\in\mathcal{T}^\omega$, let
$$ \Psi((S_n), (S'_n))=\bigcap_n \Phi(S_n, S'_n). $$
Then $\Psi((S_n), (S'_n))\subseteq T_G$ is a coset tree. $\Psi((S_n), (S'_n))$ is illfounded if and only if $((S_n), (S'_n))\in F$. For any $\alpha<\omega_1$, the set
$$\{S\in \mathcal{T}\,:\, \mbox{$S$ is a wellfounded and $\het(S)\leq\alpha$}\}$$ is Borel. Thus if there is an $\alpha<\omega_1$ such that $\het(S)\leq\alpha$ for all wellfounded coset trees $S$, then $F$ is Borel, and therefore $G$ is tame.

The proof of the equivalence of (1)--(3) is essentially the same as the proof of Lemma 6 of \cite{solecki}. (2)$\Rightarrow$(1) is trivial. For (3)$\Rightarrow$(2), use Lemma~\ref{lem:c2g} above and note that if $S\subseteq T_G$ then $\Gamma(S)\subseteq T_G$. For (2)$\Rightarrow$(3) we argue as in the proof of \cite{solecki} Lemma 6 (iii)$\Rightarrow$(ii). When Lemma 5 of \cite{solecki} is applied in that proof replace it by  Lemma~\ref{lem:g2c}. The resulting wellfounded coset tree is now a subtree of $T_G$.
\end{proof}

\section{Characterizing tameness}
Most groups in this section are assumed to be abelian. We will use the additive notation for such groups. The identity element will be denoted $0$ (with subscripts when necessary).

We will make use of the concept of $p$-compactness defined by Solecki in \cite{solecki}. Let $p$ be a prime. A group $\Gamma$ is called {\em $p$-compact} if for any decreasing sequence of groups $(G_n)$ such that for each $n\in\omega$, $G_n\leq \mathbb{Z}(p)\times \Gamma$ with $\pi[G_n]=\mathbb{Z}(p)$, where $\pi: \mathbb{Z}(p)\times\Gamma\to \mathbb{Z}(p)$ is the projection map, we have $\pi[\bigcap_n G_n]=\mathbb{Z}(p)$.

The concept of $p$-compactness makes sense for arbitrary groups, but little is known about it for nonabelian groups. Solecki gave several equivalent formulations of the concept in the abelian case, as follows.

\begin{proposition}[Solecki \cite{solecki}]\label{soleckipcompact} Let $p$ be a prime. Let $\Gamma$ be a countable abelian group. Then the following are equivalent:
\begin{enumerate}
\item[\rm (i)] $\Gamma$ is $p$-compact.
\item[\rm (ii)] $\Gamma$ is torsion, and the $p$-component of $\Gamma$ satisfies the minimal condition, i.e., there is no infinite strictly decreasing sequence of subgroups of $\Gamma$.
\item[\rm (iii)] $\Gamma$ is torsion, and the $p$-component of $\Gamma$ is of the form $F\oplus \mathbb{Z}(p^\infty)^k$, where $F$ is a finite $p$-group and $k\in\omega$.
\item[\rm (iv)] $\Gamma$ is torsion, and for any finite $p$-group $F\leq \Gamma$ the $p$-rank of $\Gamma/F$ is finite.
\item[\rm (v)] $\Gamma$ is torsion, and $\Gamma$ does not involve $\mathbb{Z}(p)^{<\omega}$.
\end{enumerate}
\end{proposition}

The main result of \cite{solecki} is the following theorem.

\begin{theorem}[Solecki \cite{solecki}] \label{soleckimain} Let $(H_n)$ be a sequence of countable discrete abelian groups. Then $\prod_n H_n$ is tame if and only if for each prime $p$, for all but finitely many $n$, $H_n$ is $p$-compact.
\end{theorem}

It follows from this theorem that $\mathbb{Z}^\omega$ and $(\mathbb{Z}(p)^{<\omega})^\omega$ for any prime $p$ are examples of non-tame groups.

In our study here we first offer two additional formulations of $p$-compactness for countable abelian groups. For any prime $p$ and abelian group $\Gamma$, let $\Gamma_p$ denote the {\em $p$-component} of $\Gamma$, i.e., $$\Gamma_p=\{g\in\Gamma\,:\, \exists k\in\omega\ p^kg=0\};$$ also denote $$\Gamma[p]=\{g\in \Gamma\,:\, pg=0\}.$$ The nonzero elements of $\Gamma[p]$ are exactly the group elements of order $p$.

\begin{lemma}\label{compact} Let $p$ be a prime. Let $\Gamma$ be a countable abelian group. Then the following are also equivalent to the $p$-compactness of $\Gamma$:
\begin{enumerate}
\item[\rm (vi)] $\Gamma$ is torsion, and $\Gamma[p]$ is finite.
\item[\rm (vii)] $\Gamma$ does not have a subgroup isomorphic to either $\Z$ or $\Z(p)^{<\omega}$.
\end{enumerate}
\end{lemma}

\begin{proof} 
Note that $\Gamma[p]$ is a vector space over $\Z(p)$. When its dimension is finite, $\Gamma[p]$ is finite. When its dimension is (countably) infinite, $\Gamma[p]$ is  isomorphic to $\mathbb{Z}(p)^{<\omega}$. From these observations it follows easily that (vi)$\Leftrightarrow$(vii). 

It is clear that (vii) follows from Proposition~\ref{soleckipcompact} (v).

To complete the proof we show that (vi) implies Proposition~\ref{soleckipcompact} (ii).  Assume that $\Gamma$ is torsion and $\Gamma[p]$ is finite. By a theorem of Kulikov (c.f. \cite{Ro} 4.3.4), there exists a subgroup $L\le \Gamma_p$ such that $L$ is a direct sum of cyclic groups and $\Gamma_p/L$ is divisible. Since $L[p]\subseteq \Gamma[p]$ is finite, we see that $L$ itself is also finite. Now a divisible $p$-group is a direct sum of $p$-quasicyclic groups (c.f. \cite{Ro} 4.1.5), and therefore either $\Gamma_p/L\cong\Z(p^\infty)^n$ for some $n\in\omega$ or $H_p/L\cong\Z(p^\infty)^{<\omega}$. If $\Gamma_p/L\cong\Z(p^\infty)^{<\omega}$, then $(\Gamma_p/L)[p]$ is infinite, thus $A=\{g\in \Gamma_p:pg\in L\}$ is also infinite. But note that $g\mapsto pg$ is a homomorphism from $A$ into $L$ with kernel $\Gamma[p]$, we have that $A/\Gamma[p]$ is isomorphic to a subgroup of $L$. Since $\Gamma[p]$ and $L$ are finite while $A$ is infinite, this is a contradiction. Therefore, $\Gamma_p/L\cong\Z(p^\infty)^n$ for some $n\in\omega$. In particular, both $\Gamma_p/L$ and $L$ satisfy the minimal condition. Since groups satisfying the minimal condition are closed under group extensions (c.f. \cite{Ro} 3.1.7), we conclude that $\Gamma_p$ satisfies the minimal condition. 
\end{proof}

We will also use the following simple corollary for countable abelian $p$-compact groups.

\begin{lemma}\label{extansion}
Let $p$ be a prime. Let $\Gamma$ be a countable abelian group and $L\leq \Gamma$. Then $\Gamma$ is $p$-compact if and only if both $L$ and $\Gamma/L$ are $p$-compact.
\end{lemma}

\begin{proof}
For ($\Rightarrow$) assume that $\Gamma$ is $p$-compact. It is trivial from the definition of $p$-compactness that $L$ is $p$-compact. Define homomorphism $\phi:\Z(p)\times \Gamma\to\Z(p)\times \Gamma/L$ as $\phi(a,g)=(a,g+L)$. If $(G_n)$ is a decreasing sequences of groups witnessing that $\Gamma/L$ is not $p$-compact, then $(\phi^{-1}(G_n))$ witnesses that $\Gamma$ is not $p$-compact.

For ($\Leftarrow$) assume that both $L$ and $\Gamma/L$ are $p$-compact. Using Proposition~\ref{soleckipcompact} (ii) it suffices to show that $\Gamma$ is torsion and $\Gamma_p$ satisfies the minimal condition. Again using Proposition~\ref{soleckipcompact} (ii) we assume that both $L$ and $\Gamma/L$ are torsion, and that both $L_p$ and $(\Gamma/L)_p$ satisfy the minimal condition. It is easy to see that $L_p=\Gamma_p\cap L$ and $\Gamma_p/L_p\subseteq (\Gamma/L)_p$. Since groups satisfying the minimal condition are closed under subgroups and group extensions (c.f. \cite{Ro} 3.1.7), we get that $\Gamma_p$ satisfies the minimal condition. It is easy to verify that torsion groups are also closed under subgroups and group extensions, and thus $\Gamma$ is also torsion. 
\end{proof}

We remark that the backward direction of Lemma~\ref{extansion} fails without the abelian assumption. A counterexample, which is solvable of rank 2, was constructed by Hjorth in \cite{Hj}. 

We are now ready to prove Theorem~\ref{main1}. For the rest of the section let $(H_n)$ be a sequence of countable discrete abelian groups and $G\leq\prod_n H_n$ be a closed subgroup. We use $\vec{0}$ to denote the identity element of $G$. We will use all notation developed in Section 3 regarding group trees. In addition, we will also use the following notation. Let  $G_{\langle 0\rangle}=G$, and for $m\in\mathbb{N}_+$, let
$$G_{\langle m\rangle}=\{x\in G:\forall k< m\ \pi_k(x)=0\}.$$
Then each $G_{\langle m\rangle}$ is an open subgroup of $G$, and $(G_{\langle m\rangle})$ is a nbhd base of $\vec{0}$ in $G$. Note that for any $n\in\omega$, $G_{\langle n\rangle}/G_{\langle n+1\rangle}\cong \pi_n[G_{\langle n\rangle}]\leq H_n$.

\begin{lemma} \label{notpcompact} Let $p$ be a prime and $m\in\mathbb{N}_+$. Let $S\subseteq T_G$ be a group tree and $\sigma\in S\cap H^m$ with a power of $p$ as its order. If $\omega\leq r(\sigma)<\omega_1$ and $r(\sigma)$ is a limit ordinal, then $G_{\langle m\rangle}/G_{\langle m+1\rangle}$ is not $p$-compact.
\end{lemma}

\begin{proof}
We will continue to use $\pi^m$ to denote the projection from $H^{m+1}$ to $H^m$. Let $(\beta_i)$ be a strictly increasing sequence of ordinals with limit $r(\sigma)$. For each $i\in\omega$, let $$K_i=\{\tau\in S\cap H^{m+1}:r(\tau)\ge\beta_i\}.$$ Then  $(K_i)$ is a decreasing sequence of subgroups of $H^{m+1}$, and $$\sigma\in\left(\displaystyle\bigcap_i\pi^m[K_i]\right)\setminus\pi^m\left[\bigcap_iK_i\right].$$ Let $C=\langle\sigma\rangle\le H^m$. Then $C\cong\Z(p^k)$ for some $k\in\omega$. Let 
$$\Gamma=\pi_m[(\pi^m)^{-1}(C)\cap S]\leq H_m$$ and for each $i\in\omega$, $$K_i'=K_i\cap(C\times\Gamma).$$ Let $\phi:C\to\Z(p)$ be a surjective homomorphism. Let $\Phi=\phi\times{\rm id}:C\times\Gamma\to\Z(p)\times\Gamma$. Since $\Phi$ is finite-to-one, we have 
$$\Phi\left[\bigcap_iK_i'\right]=\bigcap_i\Phi[K_i'].$$ Let $\pi:\Z(p)\times\Gamma\to\Z(p)$ be the projection map. We have $\pi\circ\Phi=\phi\circ\pi^m$ on $C\times \Gamma$.
Since $\sigma\notin\pi^m\left[\bigcap_iK_i\right]$, we have $$\pi^m\left[\bigcap_iK_i'\right]=\pi^m\left[\bigcap_i K_i\cap (C\times \Gamma)\right]\subseteq \pi^m\left[\bigcap_i K_i\right]\cap C \leq \langle p\sigma\rangle. $$
Thus $$\pi\left[\bigcap_i\Phi[K_i']\right]=\pi\left[\Phi\left[\bigcap_i K_i'\right]\right]=\phi\left[\pi^m\left[\bigcap_iK_i'\right]\right]=\{0\}.$$ On the other hand, for each $i\in\omega$, we have $\sigma\in\pi[K_i']$, so $\pi[\Phi[K_i']]=\phi[\pi^m[K_i']]=\Z(p)$. Thus the decreasing sequence of groups $(\Phi[K_i'])$ witnesses that $\Gamma$ is  not $p$-compact.

Now let $L=\pi_m[\ker(\pi^m)\cap S]$. Note that $L\le\Gamma$ with
$${\rm card}(\Gamma/L)\le{\rm card}(((\pi^m)^{-1}(C)\cap S)/(\ker(\pi^m)\cap S))\le{\rm card}(C)=p^k.$$
Thus $\Gamma/L$ is finite, and in particular $p$-compact. Since $\Gamma$ is not $p$-compact, it follows from Lemma \ref{extansion} that $L$ is also not $p$-compact.

In the end, note that $L\le\pi_m(G_{\langle m\rangle})\cong G_{\langle m\rangle}/G_{\langle m+1\rangle}$ and so $G_{\langle m\rangle}/G_{\langle m+1\rangle}$ is not $p$-compact.
\end{proof}

\begin{lemma}\label{nt2np}
Suppose for any $\alpha<\omega_1$ there exists a group tree $S\subseteq T_G$ with $\het(S)>\alpha$. Then for some prime $p$ there exist infinitely many $n\in\omega$ such that $G_{\langle n\rangle}/G_{\langle n+1\rangle}$ is not $p$-compact.
\end{lemma}

\begin{proof} This lemma is similar to Lemma 8 of \cite{solecki} with the proof being the same, except Lemma~\ref{notpcompact} is needed to replace a similar claim in that proof.
\end{proof}

Note that Lemma 8 of \cite{solecki} holds for countable products of countable groups without the abelian assumption. Here we need the abelian assumption since Lemma~\ref{extansion} is used.

\begin{lemma}\label{anybase} Let $p$ be any prime. Let $(G_n)$ be any decreasing sequence of open subgroups of $G$ such that $(G_n)$ is a nbhd base of $\vec{0}$. Then there are infinitely many $n\in\omega$ such that $G_{\langle n\rangle}/G_{\langle n+1\rangle}$ is not $p$-compact if and only if there are infinitely many $n\in\omega$ such that $G_n/G_{n+1}$ is not $p$-compact.
\end{lemma}

\begin{proof} Suppose there are infinitely many $k\in\omega$ such that $G_k/G_{k+1}$ is not $p$-compact. 
Fix any $n\in\omega$. There exist $k\in\omega$ and $N>n$ such that $G_k/G_{k+1}$ is not $p$-compact and $G_{\langle N\rangle}\le G_{k+1}\le G_k\le G_{\langle n\rangle}$. Since $G_k/G_{k+1}$ is not $p$-compact, neither is $G_{\langle n\rangle}/G_{\langle N\rangle}$ by Lemma~\ref{extansion}. By repetitive applications of Lemma \ref{extansion} again, we get that there is $n\leq m< N$ such that $G_{\langle m\rangle}/G_{\langle m+1\rangle}$ is not $p$-compact. This shows that there are infinitely many $m\in\omega$ such that $G_{\langle m\rangle}/G_{\langle m+1\rangle}$ is not $p$-compact. The converse direction is symmetric.
\end{proof}

We have finally come to an extended version of Theorem~\ref{main1}.

\begin{theorem}\label{main1technical}
Let $G$ be a non-archimedean abelian Polish group. The following are equivalent:
\begin{enumerate}
\item[(0)] $G$ is tame.
\item[(1)] There is a decreasing sequence of open subgroups $(G_k)$ which is a nbhd base of the identity element of $G$ such that
\begin{itemize}
\item[(a)] for all but finitely many $k$, $G_k/G_{k+1}$ is torsion, and 
\item[(b)] for any prime $p$ and for all but finitely many $k$, $G_k/G_{k+1}$ contains only finitely many elements of order $p$.
\end{itemize}
\item[(2)] For any decreasing sequence of open subgroups $(G_k)$ that is a nbhd base of the identity element of $G$, we have that
\begin{itemize}
\item[(a)] for all but finitely many $k$, $G_k/G_{k+1}$ is torsion, and 
\item[(b)] for any prime $p$ and for all but finitely many $k$, $G_k/G_{k+1}$ contains only finitely many elements of order $p$.
\end{itemize}
\item[(3)] $G$ does not contain a closed subgroup isomorphic to $\Z^\omega$, and $G$ does not involve $(\mathbb{Z}(p)^{<\omega})^\omega$ for any prime $p$.
\item[(4)] $G$ does not involve either $\Z^\omega$ or $(\mathbb{Z}(p)^{<\omega})^\omega$ for any prime $p$.
\end{enumerate}
\end{theorem}

\begin{proof}
By Proposition~\ref{prop:abel}, we may fix a sequence of countable abelian groups $(H_n)$ such that $G$ is a closed subgroup of $\prod_n H_n$. Consider the additional clause:
\begin{enumerate}
\item[(*)] {\em For any prime $p$, for all but finitely many $n\in\omega$, $G_{\langle n\rangle}/G_{\langle n+1\rangle}$ is $p$-compact.}
\end{enumerate}
The equivalences (*)$\Leftrightarrow$(1)$\Leftrightarrow$(2) follow from Lemma~\ref{anybase} and Lemma~\ref{compact} (vi).

(*)$\Rightarrow$(0) follows from Lemma~\ref{nt2np} and Proposition~\ref{prop:tamebytrees} (3).

(0)$\Rightarrow$(4) because $\mathbb{Z}^\omega$ and $(\Z(p)^{<\omega})^\omega$ are not tame for any prime $p$, and because of the general fact that if $G$ involves $H$ and $G$ is tame, then so is $H$.

(4)$\Rightarrow$(3) is obvious.

It remains to prove (3)$\Rightarrow$(*).

We assume there is a strictly increasing sequence $(n_k)$ of natural numbers such that $G_{\langle n_k\rangle}/G_{\langle n_k+1\rangle}$ is not $p$-compact for each $k\in\omega$. By Lemma~\ref{compact} (vi), either $G_{\langle n_k\rangle}/G_{\langle n_k+1\rangle}$ is not torsion or $(G_{\langle n_k\rangle}/G_{\langle n_k+1\rangle})[p]$ is infinite. We consider two cases.

{\sl Case 1.} There are infinitely many $k\in\omega$ such that $G_{\langle n_k\rangle}/G_{\langle n_k+1\rangle}$ is not torsion. Without loss of generality, we can assume that this holds for every $k\in\omega$. Note that $n_k>k$ for all $k\in\omega$. For each $k$, choose $x_k\in G_{\langle n_k\rangle}$ such that $x_k+G_{\langle n_k+1\rangle}$ has infinite order in $G_{\langle n_k\rangle}/G_{\langle n_k+1\rangle}$. Note that $\pi^{n_k}(x_k)=0$ and $\pi_{n_k}(mx_k)=m\pi_{n_k}(x_k)\ne 0$ for all $m\ne 0$. Let $K$ be the closure of $\langle x_0,x_1,\cdots\rangle$ in $G$. Then $K$ is a closed subgroup of $G$. Let $T_K$ be the group tree with $[T_K]=K$. We will complete the proof in this case with showing that $K\cong\Z^\omega$. 

Define $\phi:\Z^\omega\to K$ as $$\phi(f)=\sum_kf(k)x_k$$ for each $f\in\Z^\omega$. We first verify that $\phi(f)$ is a well-defined element of $\prod_n H_n$. For this, note that for any $n\in\omega$, letting $l\in\mathbb{N}_+$ be the least such that $n_l>n$, we have $$\pi_n(\phi(f))=\sum_kf(k)\pi_n(x_k)=\sum_{k<l}f(k)\pi_n(x_k)\in H_n.$$ So $\phi(f)\in\prod_n H_n$. We now verify that indeed $\phi(f)\in K$. For this, note similarly that for any $m\in\mathbb{N}_+$, letting $l\in\mathbb{N}_+$ be the least such that $n_l\geq m$, we have
$$ \pi^m(\phi(f))=\sum_{k<l}f(k)\pi^m(x_k)=
\pi^m\left(\sum_{k<l}f(k)x_k\right)\in T_K\cap H^m.$$
So $\phi(f)\in[T_K]=K$.

It is clear that $\phi$ is a homomorphism. We verify that $\phi$ is injective. For any   $f\ne 0$, let $i\in\omega$ be the least such that $f(i)\ne 0$. Then $\pi_{n_i}(\phi(f))=f(i)\pi_{n_i}(x_i)\ne 0$, and so $\phi(f)\neq 0$. 

Next we verify that $\phi$ is surjective. For this let $y\in K$ and let $(f_j)$ be a sequence of elements in $\Z^\omega$ with $\lim_j\phi(f_j)=y$. We can find an increasing sequence $(j_l)_{l>0}$ such that for all $j\geq j_l$, $\pi^{n_l}(\phi(f_j))=\pi^{n_l}(y)$. Thus for all $j\geq j_l$ and $n<n_l$,
$$ \pi_n(y)=\pi_n(\phi(f_j))=\sum_{k<l} f_j(k)\pi_n(x_k). $$
By induction on $k\in\omega$ we prove that $f_j(k)=f_{j'}(k)$ for all $j, j'\geq j_{k+1}$. For $k=0$ set $l=1$ and consider $\pi_{n_0}(y)$. For all $j, j'\geq j_1$, $f_j(0)\pi_{n_0}(x_0)=\pi_{n_0}(y)=f_{j'}(0)\pi_{n_0}(x_0)$. Since $\pi_{n_0}(x_0)\neq 0$, we get $f_j(0)=f_{j'}(0)$. In the inductive step, assume $k>0$ and for all $k'<k$ and $j, j'\geq j_k$, $f_j(k')=f_{j'}(k')$. Set $l=k+1$ and consider $\pi_{n_k}(y)$. Suppose $j, j'\geq j_{k+1}$. Since $j_{k+1}\geq j_k$, we have $f_j(k')=f_{j'}(k')$ for all $k'<k$. Now $\pi_{n_k}(y)$ equals both
$$\begin{array}{l} \displaystyle\sum_{k'<k} f_j(k')\pi_{n_k}(x_{k'})+f_j(k)\pi_{n_k}(x_k)=\sum_{k'<k} f_{j'}(k')\pi_{n_k}(x_{k'})+f_{j'}(k)\pi_{n_k}(x_k). 
\end{array}$$
It follows that $f_j(k)\pi_{n_k}(x_k)=f_{j'}(k)\pi_{n_k}(x_k)$. Since $\pi_{n_k}(x_k)\neq 0$, we get $f_j(k)=f_{j'}(k)$. This finishes the induction. Define $f(k)=f_{j_{k+1}}(k)$ for all $k\in\omega$. It is easy to check that $\phi(f)=y$.

Next we verify that $\phi$ is continuous. For this let $(f_j)$ tend to $\vec{0}$ in $\Z^\omega$. We can find a sequence $(j_l)$ such that for all $j\geq j_l$ and for all $k<l$, we have $f_j(k)=0$. Now fix $m\in\mathbb{N}_+$. Let $l\in\mathbb{N}_+$ be the least such that $n_l\geq m$. Then for $j\geq j_l$, we have
$$ \pi^m(\phi(f_j))=\pi^m\left(\sum_{k<l}f_j(k)x_k\right)=0. $$ This means $\lim_j\pi^m(\phi(f_j))=0$ for any $m\in\mathbb{N}_+$, and thus $\lim_j \phi(f_j)=\vec{0}$. This shows that $\phi$ is continuous. The continuity of $\phi^{-1}$ can be similarly shown. 

We have thus shown that $K$ and $\Z^\omega$ are isomorphic as topological groups.

{\sl Case 2.}  There are infinitely many $k\in\omega$ such that $(G_{\langle n_k\rangle}/G_{\langle n_k+1\rangle})[p]$ is infinite. Without loss of generality, we can assume that this holds for every $k\in\omega$. Since $G_{\langle n_k\rangle}/G_{\langle n_k+1\rangle}\cong\pi_{n_k}[G_{\langle n_k\rangle}]\le H_{n_k}$, we know that $\pi_{n_k}[G_{\langle n_k\rangle}][p]$ is infinite. It follows that, as a vector space over $\Z(p)$, $\pi_{n_k}[G_{\langle n_k\rangle}][p]$ has infinite dimension. Let $u^k_0,u^k_1,\cdots$ enumerate without repetition a basis of this vector space. Then for any $g\in\pi_{n_k}[G_{\langle n_k\rangle}][p]$, there is a unique sequence $(c_i(g))$ of elements in $\Z(p)$ such that for all but finitely many $i\in\omega$, $c_i(g)=0$, and $g=\sum_ic_i(g)u^k_i$. Call the finite set $\{i\in\omega:c_i(g)\ne 0\}$ the {\em support} of $g$, and denote it by $\supp(g)$.  For any $E\subseteq\omega$, we denote $g\!\upharpoonright\! E=\sum_{i\in E}c_i(g)u^k_i$. Then $\supp(g\!\upharpoonright\!E)\subseteq E$.

Fix a sequence of natural numbers $(k_l)$ such that $k_l\le l$ for all $l\in\omega$, and each natural number appears in it infinitely many times. For each $l\in\omega$, denote $m_l=n_{k_l}$. We will inductively find a sequence $(x_l)$ in $G$ such that for each $l\in\omega$:
\begin{enumerate}
\item[(a)] $x_l\in G_{\langle m_l\rangle}$,
\item[(b)] $0\ne \pi_{m_l}(x_l)\in\pi_{m_l}[G_{\langle m_l\rangle}][p]$, and
\item[(c)] $\supp(\pi_{n_k}(x_l))\cap\supp(\pi_{n_k}(x_{l'}))=\emptyset$ for $k\le l$ and $l'<l$.
\end{enumerate}

For $l=0$, we have $k_0=0$ and $m_0=n_0$. Choose any $x_0\in G_{\langle n_0\rangle}$ with $0\ne \pi_{n_0}(x_0)\in\pi_{n_0}[G_{\langle n_0\rangle}][p]$. For the inductive step, assume $l>0$ and suppose that $x_0,\cdots, x_{l-1}\in G$ have been chosen. For each $k\le l$, let $E_k=\bigcup_{l'<l}\supp(\pi_{n_k}(x_{l'}))$. Note that each $E_k$ is finite, while $\pi_{n_k}[G_{\langle n_k\rangle}][p]$ is infinite. By the pigeonhole principle, we can find $y,y'\in G_{\langle m_l\rangle}$ such that $\pi_{m_l}(y), \pi_{m_l}(y')\in\pi_{m_l}[G_{\langle m_l\rangle}][p]$ with $\pi_{m_l}(y)\ne \pi_{m_l}(y')$, but $\pi_{n_k}(y)\!\upharpoonright\!E_k=\pi_{n_k}(y')\!\upharpoonright\!E_k$ for each $k\le l$. Define $x_l=y-y'$. It is routine to check that $x_l$ satisfies clauses (a)--(c). This finishes the inductive construction.

For each $k\in\omega$, define $K_k=\langle\{\pi_{n_k}(x_l)\,:\,k_l=k\}\rangle$. For $l'<l$, if $k=k_{l'}=k_l$, we have $k=k_l\le l$, and $\supp(\pi_{n_k}(x_l))\cap\supp(\pi_{n_k}(x_{l'}))=\emptyset$ by (c). So $K_k=\bigoplus_{\{l:k_l=k\}}\langle \pi_{n_k}(x_l)\rangle\cong\bigoplus_\omega\Z(p)=\Z(p)^{<\omega}$.

For each $k\in\omega$, define $F_k=\bigcup_{g\in K_k}\supp(g)$. It follows from (c) that for any distinct $l,l'$ with $k_l=k$, we always have $\supp(\pi_{n_k}(x_l))\cap\supp(\pi_{n_k}(x_{l'}))=\emptyset$. In particular, if $k_{l'}\neq k$, then $\pi_{n_k}(x_{l'})\!\upharpoonright\!F_k=0$.

We now define two closed subgroups of $G$:
$$\begin{array}{l}
A=\{x\in G\,:\,\forall k\ \pi_{n_k}(x)\!\upharpoonright\! F_k\in K_k)\}, \\ \\
A_0=\{x\in G\,:\,\forall k\ \pi_{n_k}(x)\!\upharpoonright\!F_k=0)\}.
\end{array}$$
We will complete the proof in this case with showing  $$A/A_0\cong\prod_kK_k\cong(\Z(p)^{<\omega})^\omega.$$ 

Define $\psi:A\to\prod_kK_k$ as 
$$\psi(x)(k)=\pi_{n_k}(x)\!\upharpoonright\! F_k$$ for $x\in A$ and $k\in\omega$. It easy to see that $\psi$ is a continuous homomorphism with $\ker(\psi)=A_0$. Here we only verify that $\psi$ is surjective. For this let $y\in\prod_kK_k$. For any $k\in\omega$, we can write $$y(k)=\sum_{l\in I_k}c^k_l\pi_{n_k}(x_l)$$ where $I_k$ is a finite set of $l$ with $k_l=k$ and each $c^k_l\in\Z(p)$. Note that $\sum_{l\in I_k}c^k_lx_l\in G_{\langle n_k\rangle}$, so $\sum_k\sum_{l\in I_k}c^k_lx_l$ converges to an element $x\in G$. For any $k\in\omega$,
$$\pi_{n_k}(x)\!\upharpoonright\! F_k=\sum_{l\in I_k}c^k_l\pi_{n_k}(x_l)=y(k)\in K_k.$$
It follows that $x\in A$ and $\psi(x)=y$. 
\end{proof}

The following is a corollary of the proof of Theorem~\ref{main1technical}.

\begin{theorem} Let $G$ be a non-archimedean abelian Polish group and $(G_n)$ be any decreasing sequence of open subgroups that is a nbhd base of the identity element of $G$. Then the following are equivalent:
\begin{enumerate}
\item[\rm (i)] $G$ involves $\mathbb{Z}^\omega$.
\item[\rm (ii)] $G$ has a closed subgroup isomorphic to $\Z^\omega$.
\item[\rm (iii)] For infinitely many $n$, $G_n/G_{n+1}$ is not torsion.
\end{enumerate}
\end{theorem}

\begin{proof}
The implication (ii)$\Rightarrow$(i) is obvious. The implication (iii)$\Rightarrow$(ii) follows from the proof of Lemma~\ref{anybase} and the construction in Case 1 of the proof of Theorem~\ref{main1technical}. It remains to prove (i)$\Rightarrow$(iii). Without loss of generality, we may assume that there is a continuous homomorphism $\phi$ from $G$ onto $\Z^\omega$. Note that $\phi$ is also a quotient map from $G$ onto $\Z^\omega$ as topological spaces. It follows that, for each $n\in\omega$, $\phi[G_n]$ is open in $\Z^\omega$, since $\phi^{-1}(\phi[G_n])=G_n\ker(\phi)$ is open in $G$. For each $n\in\omega$, let $Z_n=(\Z^\omega)_{\langle n\rangle}$. Then   $(Z_n)$ is a decreasing sequence of open subgroups of $\Z^\omega$ that is also a nbhd base of the identity element of $\Z^\omega$. Thus, for each $n\in\omega$, there exist $k\in\omega$ and $N>n$ such that $\phi[G_N]\leq Z_{k+1}\leq Z_k\leq \phi[G_n]$. Since $Z_k/Z_{k+1}$ is not torsion, it follows that there is $n\leq m<N$ such that $\phi[G_m]/\phi[G_{m+1}]$ is not torsion. But $\phi[G_m]/\phi[G_{m+1}]$ is a homomorphic image of $G_m/G_{m+1}$, and therefore $G_m/G_{m+1}$ is not torsion. 
\end{proof}

Recently Malicki \cite{Ma2} also studied tame non-archimedean abelian Polish groups. Among other things, his main results gave a characterization of tameness in the quasi-reduced case which is derivable from our Theorem~\ref{main1technical}.

\section{Orbit equivalence relations of tame groups}

In this section we study the orbit equivalence relations induced by actions of tame non-archimedean abelian Polish groups. We will first give some upper bounds on the height of the group trees and coset trees of such groups, and then obtain upper bounds on the descriptive complexity of the orbit equivalence relations.

We will concentrate on a special representation of tame non-archimedean abelian Polish groups as closed subgroups of product groups.

\begin{lemma}\label{rearrangement} Let $G$ be a tame non-archimedean abelian Polish group. Then there is a sequence $(H_n)$ of countable discrete abelian groups such that $G$ is isomorphic to a closed subgroup of $\prod_n H_n$, and, as such, $G$ satisfies that for all $n\in\mathbb{N}_+$, $G_{\langle n\rangle}/G_{\langle n+1\rangle}$ is torsion.
\end{lemma}

\begin{proof}
By Proposition~\ref{prop:abel}, there is a sequence $(\Gamma_n)$ of countable discrete abelian groups such that $G\leq \prod_n\Gamma_n$ is a closed subgroup. Since $G$ is tame, by Theorem~\ref{main1technical} we have that for all but finitely many $n\in\omega$, $G_{\langle n\rangle}/G_{\langle n+1\rangle}$ is torsion. Note that $G_{\langle n\rangle}$ here is defined as a subgroup of $\prod_n \Gamma_n$. 

Let $m\in\omega$ be the largest $n$ such that $G_{\langle n\rangle}/G_{\langle n+1\rangle}$ is not torsion, if such $n$ exists; and let $m=0$ otherwise. Define $H_0=\Gamma_0\times \dots\times \Gamma_m$ and for $n\in\mathbb{N}_+$, define $H_n=\Gamma_{m+n}$. Then $\prod_n \Gamma_n$ and $\prod_n H_n$ are isomorphic topological groups, and $G$ is still isomorphic to a closed subgroup of $\prod_n H_n$.

As a closed subgroup of $\prod_n H_n$, $G$ satisfies that $G_{\langle n\rangle}/G_{\langle n+1\rangle}$ is torsion for all $n\in\mathbb{N}_+$.
\end{proof}

We will need the following simple property guaranteed by this kind of representation.

\begin{lemma} \label{finiteorderextension} Let $(H_n)$ be a sequence of countable discrete abelian groups and $G\leq \prod_n H_n$ be a closed subgroup satisfying that for all $n\in\mathbb{N}_+$, $G_{\langle n\rangle}/G_{\langle n+1\rangle}$ is torsion. Let $S\subseteq T_G$ be any group tree and $\sigma\in S$ with finite order. Then any extension $\tau\supseteq \sigma$ in $S$ also has finite order.
\end{lemma}

\begin{proof} Toward a contradiction, assume that $\tau\supseteq \sigma$ has infinite order and has minimal length among such extensions. Without loss of generality, we may assume $\lh(\tau)=\lh(\sigma)+1$. Suppose $\lh(\sigma)=m\in\mathbb{N}_+$ and $a\in\mathbb{N}_+$ is the order of $\sigma$. Since $S$ is a group tree, we have $k\tau\supseteq k\sigma$ for all $k\in\mathbb{Z}$. In particular, $a\tau\supseteq a\sigma=0\in S\cap H^m\subseteq T_G\cap H^m$. Note that $a\tau\in T_G\cap H^{m+1}$ still has infinite order. Let $x\in G=[T_G]$ be any element with $x\supseteq a\tau$. Then $x\in G_{\langle m\rangle}$, $x$ has infinite order, and for any nonzero $k\in\mathbb{Z}$, $kx\not\in G_{\langle m+1\rangle}$. This implies that $G_{\langle m\rangle}/G_{\langle m+1\rangle}$ is not torsion, contrary to our assumption. 
\end{proof}

In our proofs below we will make use of the following facts proved in \cite{solecki}.

\begin{lemma}[Solecki \cite{solecki}]\label{lemma81} Let $S$ be any group tree and $\sigma\in S$ with finite order and $r(\sigma)=\alpha<\omega_1$. Then there is a prime $p$ and $\tau\in S$ such that the order of $\tau$ is a power of $p$ and $\alpha\leq r(\tau)<\omega_1$.
\end{lemma}

\begin{proof}
The proof is identical to the penultimate paragraph in the proof of Lemma 8 in \cite{solecki}.
\end{proof}

\begin{lemma}[Solecki \cite{solecki}] \label{lemma82} Let $p$ be a prime. Let $S$ be a group tree. Let $\sigma\in S$ satisfy that $r(\sigma)<\omega_1$, the order of $\sigma$ is a power of $p$, and any extension $\tau\supseteq \sigma$ in $S$ also has finite order. Then for any $\alpha<r(\sigma)$ there is $\tau\supseteq \sigma$ such that $r(\tau)=\alpha$ and the order of $\tau$ is also a power of $p$.
\end{lemma}

\begin{proof} The proof is identical to the second paragraph in the proof of Lemma 8 in \cite{solecki}.
\end{proof}

\begin{theorem}\label{rankbound} Let $(H_n)$ be a sequence of countable discrete abelian groups and let $G\leq \prod_n H_n$ be a tame closed subgroup such that for all $n\in\mathbb{N}_+$, $G_{\langle n\rangle}/G_{\langle n+1\rangle}$ is torsion. Then for any group tree $S\subseteq T_G$, $\het(S)\leq \omega\cdot 3$; for any coset tree $S\subseteq T_G$, $\het(S)<\omega\cdot 4$.
\end{theorem}

\begin{proof} Toward a contradiction, assume $S\subseteq T_G$ is a group tree with $\het(S)> \omega\cdot 3$. Let $\sigma\in S$ be such that $r(\sigma)=\omega\cdot 3$. We consider two cases.

{\sl Case 1.} The order of $\sigma$ is finite. In particular we have an element of $S$ whose order is finite and whose rank is $\geq \omega\cdot 2$. This will turn out to be sufficent to derive a contradiction. By Lemma~\ref{finiteorderextension} any extension in $S$ of an element of finite order in $S$ also has finite order. By Lemmas~\ref{lemma81} and \ref{lemma82} there exist a prime $p$ and $\tau\in S$ such that the order of $\tau$ is a power of $p$ and $r(\tau)=\omega\cdot 2$. We claim that there are infinitely many $m\in\omega$ such that there exists $\sigma_m\in S\cap H^m$ so that $r(\sigma_m)=\omega$ and the order of $\sigma_m$ is a power of $p$. Granting the claim, by Lemma~\ref{notpcompact} $G_{\langle m\rangle}/G_{\langle m+1\rangle}$ is not $p$-compact for infinitely many $m$, which by Theorem~\ref{main1technical} contradicts our assumption that $G$ is tame.

To prove the claim, fix an arbitrary $n\in\omega$. Apply Lemma~\ref{lemma82} to find $\tau_0\supseteq \tau$ in $S$ so that $r(\tau_0)=\omega+n$ and the order of $\tau_0$ is a power of $p$. Repeatedly applying Lemma~\ref{lemma82} again, we can inductively define a sequence $\tau_0\subseteq \tau_1\subseteq \dots \subseteq\tau_n$ such that for each $i\leq n$, $r(\tau_i)=\omega+n-i$ and the order of $\tau_i$ is a power of $p$. Then $\lh(\tau_n)>n$, $r(\tau_n)=\omega$, and the order of $\tau_n$ is a power of $p$, as required.

{\sl Case 2.} The order of $\sigma$ is infinite. Consider
$$ S_\sigma=\{\tau\in S\,:\, \mbox{either $\tau\subseteq \sigma$ or $\sigma\subseteq \tau$}\}. $$
Then $S_\sigma\subseteq S\subseteq T_G$ is a coset tree. Moreover, $r_{S_\sigma}(\sigma)=r_S(\sigma)=\omega\cdot 3$, hence $\het(S_\sigma)>\omega\cdot 3$. Now $\Gamma(S_\sigma)$ is a group tree, $\Gamma(S_\sigma)\subseteq S\subseteq T_G$, and by Lemma~\ref{lem:c2g}, $\het(\Gamma(S_\sigma))\geq \omega\cdot 3$. Note that $0=-\sigma+\sigma\in \Gamma(S_\sigma)$ is an element of order $1$. By Lemma~\ref{finiteorderextension} every element of $\Gamma(S_\sigma)$ has finite order. In particular there is an element of $\Gamma(S_\sigma)$ whose order is finite and whose rank is $\omega\cdot 2$. We are now in the same situation as in Case 1, which led to a contradiction.

The second part of the theorem, which concerns the height of a coset tree, follows from Lemma~\ref{lem:c2g}.
\end{proof}

\begin{corollary}\label{generalrank} Let $(H_n)$ be a sequence of countable discrete abelian groups and $G\leq \prod_n H_n$ a tame closed subgroup. Then for any group tree or coset tree $S\subseteq T_G$, $\het(S)<\omega\cdot 4$.
\end{corollary}

\begin{proof}
Combining a rearrangement in the proof of Lemma~\ref{rearrangement} and Theorem~\ref{rankbound}, we get that any group tree $S\subseteq T_G$ has height $<\omega\cdot 4$. The bound for coset trees follows from Lemma~\ref{lem:c2g}.
\end{proof}

\begin{corollary}\label{alltorsion} Let $(H_n)$ be a sequence of countable discrete abelian groups and $G\leq\prod_n H_n$ a tame closed subgroup. Suppose for all $n\in\omega$, $G_{\langle n\rangle}/G_{\langle n+1\rangle}$ is torsion. Then for any group tree $S\subseteq T_G$, $\het(S)\leq \omega\cdot 2$; for any coset tree $S\subseteq T_G$, $\het(S)<\omega\cdot 3$.
\end{corollary}

\begin{proof}
Our assumption implies that every element of $S$ has finite order. Case 1 of the proof of Theorem~\ref{rankbound} can be applied to get this corollary.
\end{proof}

For many examples of tame groups, the rank of their group trees and coset trees can be even lower.

\begin{theorem}\label{allpcompact} Let $(H_n)$ be a sequence of countable discrete abelian groups and $G\leq \prod_n H_n$ be a closed subgroup. Suppose for all prime $p$ and for all $n\in\omega$, $G_{\langle n\rangle}/G_{\langle n+1\rangle}$ is $p$-compact. Then for any group tree $S\subseteq T_G$, $\het(S)\leq\omega$; for any coset tree $S\subseteq T_G$, $\het(S)<\omega\cdot 2$.
\end{theorem}

\begin{proof}
Toward a contradiction assume $S\subseteq T_G$ is a group tree with $\het(S)>\omega$. Our assumption implies that every element of $S$ has finite order. By Lemmas ~\ref{lemma81} and \ref{lemma82} there is a prime $p$ and $\sigma\in S$ such that  $r(\sigma)=\omega$ and the order of $\sigma$ is a power of $p$. This contradicts Lemma~\ref{notpcompact} since it implies that $G_{\langle n\rangle}/G_{\langle n+1\rangle}$ is not $p$-compact for $n=\lh(\sigma)$. The second part of the theorem follows from Lemma~\ref{lem:c2g}.
\end{proof}

Examples to which the above theorem applies include $\mathbb{Z}(p^\infty)^\omega$ for any prime $p$ and $\prod_{p\in\mathbb{P}}\mathbb{Z}(p^\infty)$. 

We are now ready to move on to classifying orbit equivalence relations of tame groups. For our general discussion we fix a sequence $(H_n)$ of countable discrete abelian groups and $G\leq \prod_n H_n$ a tame closed group. Let $\mathcal{T}$ be the space of all trees $S\subseteq T_G$. From a theorem of Becker and Kechris (c.f. \cite{GaoBook} Theorem 3.3.4) the $G$-orbit equivalence relation on $\mathcal{T}^\omega$ is a universal $G$-orbit equivalence relation. 

We first consider the action of $G$ on $\mathcal{T}$. Let $E_G$ denote this orbit equivalence relation.

For each $m\in\mathbb{N}_+$, $T_G\cap H^m$ is a countable discrete abelian group. We use $\mathcal{T}\cap H^m$ to denote the collection of all $S\cap H^m$ for $S\in \mathcal{T}$. Consider the action of $T_G\cap H^m$ on
$$ \mathcal{P}_m=(\mathcal{T}\cap H^1)\times \dots \times (\mathcal{T}\cap H^m) $$
defined as
$$ \sigma\cdot (S_1\cap H^1, \dots, S_m\cap H^m)=(\sigma\!\upharpoonright\!1+(S_1\cap H^1), \dots, \sigma\!\upharpoonright\!m+(S_m\cap H^m)) $$
where in fact $\sigma\!\upharpoonright\!m=\sigma$. By a theorem of Gao and Jackson \cite{GJ}, the orbit equivalence relation of any countable abelian group is hyperfinite. Thus we can fix a Borel map 
$$\eta_m: \mathcal{P}_m\to 2^\omega $$
such that for any $S_1, \dots, S_m, S'_1,\dots, S'_m\in \mathcal{T}$, $(S_1\cap H^1,\dots, S_m\cap H^m)$ and $(S'_1\cap H^1, \dots, S'_m\cap H^m)$ are orbit equivalent if and only if 
$$\eta_m(S_1\cap H^1,\dots, S_m\cap H^m)E_0\eta_m(S'_1\cap H^1,\dots, S'_m\cap H^m). $$
Define $\varphi_m:\mathcal{T}\cap H^m\to 2^\omega$ as 
$$ \varphi_m(S\cap H^m)=\eta_m(S\cap H^1,\dots, S\cap H^m) $$
for $S\in \mathcal{T}$. Let $\varphi: \mathcal{T}\to (2^\omega)^\omega$ be defined as
$$ \varphi(S)=(\varphi_1(S\cap H^1), \dots, \varphi_m(S\cap H^m), \dots). $$

For any $S, S'\in \mathcal{T}$, let 
$$ \Phi(S, S')=\{ \sigma\in T_G\,:\, \forall 0<m\leq \lh(\sigma)\ (\sigma\!\upharpoonright\!m)+(S\cap H^m)=S'\cap H^m\}. $$
Each $\Phi(S, S')$ is a coset tree. If $\mathcal{S}\subseteq \mathcal{T}$, we denote by $\Phi(\mathcal{S})$ the collection of all $\Phi(S, S')$ for $S, S'\in \mathcal{S}$.

\begin{lemma}\label{finiterank} Let $\mathcal{S}\subseteq \mathcal{T}$. Suppose for any coset tree $T\in \Phi(\mathcal{S})$, $\het(T)<\omega$. Then for any $S, S'\in \mathcal{S}$, $SE_GS'$ if and only if $\varphi(S)E_0^\omega \varphi(S')$.
\end{lemma}

\begin{proof}
Let $S, S'\in\mathcal{S}$. First suppose $SE_GS'$, i.e., there is $x\in G$ such that $x+S=S'$. Then for each $m\in\mathbb{N}_+$, $(x\!\upharpoonright\!m)+(S\cap H^m)=S'\cap H^m$. For each $m\in\mathbb{N}_+$, by the definition of $\varphi_m$, we have $\varphi_m(S\cap H^m)E_0 \varphi_m(S'\cap H^m)$. Hence $\varphi(S)E_0^\omega \varphi(S')$. 

Conversely, suppose $\varphi(S)E_0^\omega\varphi(S')$. Then for each $m\in\mathbb{N}_+$, there exists $\sigma_m\in T_G\cap H^m$ such that for any $0<i\leq m$,  $\sigma_m\!\upharpoonright\!i+(S\cap H^i)=S'\cap H^i$. In other words, $\sigma_m\in \Phi(S, S')$. We claim that $\Phi(S, S')$ is illfounded. Granting the claim, since $\Phi(S, S')\subseteq T_G$, we get $x\in [T_G]=G$ with $x+S=S'$, hence $SE_G S'$. 

To show the claim, assume $\Phi(S, S')$ is wellfounded. By our assumption $\het(\Phi(S, S'))<\omega$, and thus there is $N\in\omega$ such that $\lh(\sigma)\leq N$ for every $\sigma\in \Phi(S, S')$. This contradicts the existence of $\sigma_m\in\Phi(S, S')$ for $m>N$.
\end{proof}

Next we consider the action of $G$ on $\mathcal{T}^\omega$. We denote this orbit equivalence relation by $\tilde{E}_G$. For each $m\in\mathbb{N}_+$ consider the diagonal action of $T_G\cap H^m$ on $(\mathcal{P}_m)^\omega$. This is still an action of a countable discrete abelian group, and so the orbit equivalence relation is hyperfinite. Let $\tilde{\eta}_m: (\mathcal{P}_m)^\omega\to 2^\omega$ be a Borel reduction of the orbit equivalence relation to $E_0$. Define $\tilde{\varphi}_m: (\mathcal{T}\cap H^m)^\omega\to 2^\omega$ and $\tilde{\varphi}: \mathcal{T}^\omega\to (2^\omega)^\omega=2^{\omega\times\omega}$ similarly as before:
$$ \tilde{\varphi}_m((S_n\cap H^m))=\tilde{\eta}_m((S_n\cap H^1, \dots, S_n\cap H^m)) $$
and
$$ \tilde{\varphi}((S_n))=(\tilde{\varphi}_m((S_n\cap H^m)_{n\in\omega})) _{m\in\mathbb{N}_+}.$$

For $(S_n), (S'_n)\in \mathcal{T}^\omega$, let
$$\Psi ((S_n), (S'_n))=\bigcap_n \Phi(S_n, S'_n). $$
Then $\Psi((S_n), (S'_n))$ is still a coset tree. If $\mathcal{S}\subseteq \mathcal{T}^\omega$, we let $\Psi(\mathcal{S})$ denote the collection of all $\Psi((S_n), (S'_n))$ for $(S_n), (S'_n)\in \mathcal{S}$.

\begin{lemma}\label{seqfiniterank} Let $\mathcal{S}\subseteq \mathcal{T}^\omega$. Suppose for any coset tree $T\in\Psi(\mathcal{S})$, $\het(T)<\omega$. Then for any $(S_n), (S'_n)\in \mathcal{S}$, $(S_n)\tilde{E}_G(S'_n)$ if and only if $\tilde{\varphi}((S_n))E_0^\omega\tilde{\varphi}((S'_n))$.
\end{lemma}

\begin{proof} The proof is identical to the proof of Lemma~\ref{finiterank}. Note that the $(\Rightarrow)$ direction does not require the assumption on the height of coset trees.
\end{proof}

We need to use the equivalence relation $(E_0^\omega)^+$ when the coset trees may have infinite rank. For any tree $S\in \mathcal{T}$ and $\sigma\in S$, define
$$ S_\sigma=\{\tau\in S\,:\, \mbox{either $\tau\subseteq \sigma$ or $\sigma\subseteq \tau$}\}. $$
For each $m\in\mathbb{N}_+$, we fix once and for all a linear ordering $<_m$ of all elements of $T_G\cap H^m$ so that the order type of $<_m$ is $\omega$, and for convenience use the convention that $0\in H^m$ is the $<_m$-least. The linear ordering $<_m$ gives rise to an enumeration function $\gamma_m: \omega\to T_G\cap H^m$ so that $\gamma_m(0)=0\in H^m$ and for all $i, j\in\omega$, $i<j$ if and only if $\gamma_m(i)<_m\gamma_m(j)$.

For each $m\in\mathbb{N}_+$, we define a map $\hat{\psi}_m: \mathcal{T}\to (2^{\omega\times\omega})^\omega$ as follows. Given $S\in \mathcal{T}$ and $i\in\omega$, define 
$$\hat{\psi}_m(S)(i)=\tilde{\varphi}((S_{\sigma_n})) $$
where $\sigma_0, \sigma_1, \dots, \sigma_n \dots$ is an enumeration of all elements of $S\cap H^m$ such that 
$$\gamma_m(i)+\sigma_0<_m\gamma_m(i)+\sigma_1<_m\dots<_m\gamma_m(i)+\sigma_n<_m\dots . $$

It is perhaps worthwhile to explain the the definition in slightly plainer language. When $m=1$ and $i=0$, $\hat{\psi}_1(S)(0)$ is the result of the following procedure. First enumerate all elements of $S\cap H^1$ (elements of first level in $S$) in the $<_1$-increasing order, getting $\sigma_0, \sigma_1, \dots, \sigma_n, \dots$. Next split the tree $S$ into infinitely many (disjoint) subtrees $S_{\sigma_0}, S_{\sigma_1}, \dots, S_{\sigma_n}, \dots$. Finally code the orbit equivalence of the sequence of trees using the coding map $\tilde{\varphi}$ defined above. This yields the invariant $\hat{\psi}_1(S)(0)$. In general, follow a similar three-step procedure for given $m$ and $i$. First enumerate all elements of $S\cap H^m$ (elements on the $m$-th level in $S$) according to the $<_m$-increasing order of $\gamma_m(i)+(S\cap H^m)$, getting $\sigma_0, \sigma_1, \dots, \sigma_n, \dots$. Next split the tree $S$ into infinitely many subtrees $S_{\sigma_0}, S_{\sigma_1}, \dots, S_{\sigma_n}, \dots$. Finally code the orbit equivalence of the sequence of trees using the coding map $\tilde{\varphi}$ defined above to yield the invariant $\hat{\psi}_m(S)(i)$. 

\begin{lemma}\label{rankbelowomega2} Let $\mathcal{S}\subseteq \mathcal{T}$. Suppose for any coset tree $T\in \Phi(\mathcal{S})$, $\het(T)<\omega\cdot 2$. Then the following are equivalent for any $S, S'\in\mathcal{S}$:
\begin{enumerate}
\item[\rm (i)] $SE_G S'$.
\item[\rm (ii)] For all $m\in\mathbb{N}_+$ there exist $i, i'\in\omega$ such that $\hat{\psi}_m(S)(i)E_0^\omega \hat{\psi}_m(S')(i')$.
\item[\rm (iii)] For all $m\in\mathbb{N}_+$, both
\begin{itemize}
\item[\rm (a)] for any $i\in\omega$ there is $i'\in\omega$ such that $\hat{\psi}_m(S)(i)E_0^\omega\hat{\psi}_m(S')(i')$ and
\item[\rm (b)] for any $i'\in\omega$ there is $i\in\omega$ such that  $\hat{\psi}_m(S)(i)E_0^\omega\hat{\psi}_m(S')(i')$.
\end{itemize}
\end{enumerate}
\end{lemma}

\begin{proof} We show (i)$\Rightarrow$(iii)$\Rightarrow$(ii)$\Rightarrow$(i). For (i)$\Rightarrow$(iii) suppose $SE_GS'$ and so there is $x\in G$ such that for all $m\in\mathbb{N}_+$, $(x\!\upharpoonright\! m)+(S\cap H^m)=S'\cap H^m$. Fix any $m\in\mathbb{N}_+$. By symmetry, it suffices to verify only (a). For this suppose $i\in\omega$ is given. Define $i'\in\omega$ to be such that
$$ \gamma_m(i')=\gamma_m(i)-(x\!\upharpoonright\! m). $$ 
Then 
$$ \gamma_m(i')+(S'\cap H^m)=\gamma_m(i)+(S\cap H^m). $$
Thus the first steps of the constructions of $\hat{\psi}_m(S)(i)$ and $\hat{\psi}_m(S')(i')$ end up with an enumeration $\sigma_0, \sigma_1, \dots, \sigma_n, \dots$ of elements of $S\cap H^m$ and an enumeration $\sigma'_0,\sigma'_1,\dots, \sigma'_n,\dots$ of elements of $S'\cap H^m$ with $(x\!\upharpoonright\!m)+\sigma_n=\sigma'_n$ for all $n\in\omega$. It follows that $x\in [\Psi((S_{\sigma_n}), (S'_{\sigma'_n}))]$. Thus by the $(\Rightarrow)$ direction of Lemma~\ref{seqfiniterank}, we have $\tilde{\varphi}((S_{\sigma_n}))E_0^\omega \tilde{\varphi}((S'_{\sigma'_n}))$ as required.

(iii)$\Rightarrow$(ii) is obvious.

For (ii)$\Rightarrow$(i), suppose $S, S'\in \mathcal{S}$ and $\het(\Phi(S, S'))\leq \omega+m-1$ for some $m\in\mathbb{N}_+$. Suppose $i, i'\in\omega$ are given such that $\hat{\psi}_m(S)(i)E_0^\omega\hat{\psi}_m(S')(i')$. Suppose $\sigma_0, \dots, \sigma_n,\dots$ enumerate $S\cap H^m$ according to the $<_m$-order of $\gamma_m(i)+(S\cap H^m)$ and $\sigma'_0, \dots, \sigma'_n, \dots$ enumerate $S'\cap H^m$ according to the $<_m$-order of  $\gamma_m(i')+(S'\cap H^m)$. We have $\tilde{\varphi}((S_{\sigma_n}))E_0^\omega \tilde{\varphi}((S'_{\sigma'_n}))$. In particular, $$\tilde{\varphi}_m((S_{\sigma_n}\cap H^m))E_0 \tilde{\varphi}_m((S'_{\sigma'_n}\cap H^m)).$$
However, for any $0<j\leq m$, $S_{\sigma_n}\cap H^j=\{\sigma_n\!\upharpoonright\!j\}$ and $S'_{\sigma'_n}\cap H^j=\{\sigma'_n\!\upharpoonright\!j\}. $ Thus there is a unique $\tau\in T_G\cap H^m$ such that for all $n\in\omega$, $\tau+\sigma_n=\sigma'_n$ and for any $0<j\leq m$, $(\tau\!\upharpoonright\! j)+(S_{\sigma_n}\cap H^j)=S'_{\sigma'_n}\cap H^j$. Let $T=\Psi((S_{\sigma_n}), (S'_{\sigma'_n}))$. We have just shown that for all $0<j\leq m$, $T\cap H^j=\{\tau\!\upharpoonright\!j\}$. 

We claim that $T=\Phi(S, S')_\tau$. For this we only need to verify that for all $l>m$, $T\cap H^l=\Phi(S, S')_\tau\cap H^l$. First let $\sigma\in T$ with $\lh(\sigma)=l>m$. Then $\sigma\supseteq\tau$. Since for all $j\in\omega$, $S\cap H^j=\bigcup_n (S_{\sigma_n}\cap H^j)$ and $S'\cap H^j=\bigcup_n (S'_{\sigma'_n}\cap H^j)$, we have that for all $j\leq l$, $(\sigma\!\upharpoonright\!j)+(S\cap H^j)=S'\cap H^j$. This means $\sigma\in \Phi(S, S')_\tau$. Conversely, if $\sigma\supseteq \tau$ and  $\sigma\in\Phi(S,S')$, then since $\tau+\sigma_n=\sigma'_n$ for each $n\in\omega$, we must have for each $n\in\omega$ and $j\leq l$, $(\sigma\!\upharpoonright\!j)+(S_{\sigma_n}\cap H^j)=S'_{\sigma'_n}\cap H^j$. This means that $\sigma\in T$.

From $\tilde{\varphi}((S_{\sigma_n}))E_0^\omega \tilde{\varphi}((S'_{\sigma'_n}))$, we also get that for any $m'>m$, $T\cap H^{m'}\neq\emptyset$.

Now if $T$ is illfounded then so is $\Phi(S, S')$ since $T\subseteq \Phi(S, S')$, and it follows that $SE_G S'$, and we are done. Suppose $T$ is wellfounded. Since $\lh(\tau)=m$ and $\het(\Phi(S, S'))<\omega+m$, we conclude that $r(\tau)<\omega$ and it follows that $\het(T)=\het(\Phi(S, S')_\tau)<\omega$. Thus there is $N\in\omega$ such that $\lh(\sigma)\leq N$ for every $\sigma\in T$. This contradicts the observation that $T\cap H^{m'}\neq\emptyset$ for $m'>N$.
\end{proof}

We need to see that the equivalence relation used to reduce the orbit equivalence relation in Lemma~\ref{rankbelowomega2} (iii) is essentially $(E_0^\omega)^+$. Let $F$ denote this equivalence relation on $X=(2^{\omega\times\omega})^{\omega\times \omega}$. Then for $x, y\in X$, 
$$ xFy\iff \forall m\ [\forall i\exists j\ x(m, i)E_0^\omega y(m, j) \mbox{ and } \forall j\exists i\ x(m, i)E_0^\omega y(m, j)]. $$
Thus $F=((E_0^\omega)^+)^\omega$. It suffices to show that $F\leq_B (E_0^\omega)^+$. For this we show a general statement. 

For equivalence relations $R_1$ on $X_1$ and $R_2$ on $X_2$, let $R_1\times R_2$ be the equivalence relation on $X_1\times X_2$ defined as
$$ (x_1, x_2)(R_1\times R_2)(y_1, y_2) \iff x_1R_1 y_1 \mbox{ and } x_2R_2y_2. $$
Denote by $\mbox{id}(\omega)$ the identity relation on $\omega$.

\begin{lemma}\label{plusomega} Let $E$ be an equivalence relation on a Polish space $X$. If $\mbox{\rm id}(\omega)\times E\leq_B E$, then $(E^+)^\omega\leq_B E^+$.
\end{lemma}

\begin{proof} It is easy to see that if $R_1\leq_B R_2$ then $R_1^+\leq_B R_2^+$ (c.f. \cite{GaoBook} Exercise 8.3.2). Thus it suffices to show $(E^+)^\omega\leq_B (\mbox{id}(\omega)\times E)^+$. For this we need to define a Borel reduction $f$ from $X^{\omega\times\omega}$ to $(\omega\times X)^\omega$. Fix a bijection $\langle \cdot,\cdot\rangle:\omega\times\omega\to\omega$. Given $(x_{m,i})\in X^{\omega\times\omega}$ and $n, j\in\omega$, let
$$ f((x_{m,i}))(\langle n, j\rangle)=(n, x_{n, j}).$$
It is straightforward to check that $f$ is a reduction from $(E^+)^\omega$ to $(\mbox{id}(\omega)\times E)^+$.
\end{proof}

To consider the action of $G$ on $\mathcal{T}^\omega$ we use the same coding technique before Lemma~\ref{seqfiniterank}. 

Fix a bijection $\langle \cdot,\cdot\rangle:\omega\times\omega\to\omega$. Let $(\cdot)_0$ and $(\cdot)_1$ be the decoding functions. That is, for any $n\in\omega$, $\langle (n)_0,(n)_1\rangle=n$. 

For each $m\in\mathbb{N}_+$, define a map $\tilde{\psi}_m: \mathcal{T}^\omega\to (2^{\omega\times\omega})^\omega$ as follows. Let a sequence $(S_n)\in\mathcal{T}^\omega$ and $i\in\omega$ be given. For each $n\in\omega$, enumerate the elements of $S_n\cap H^m$ as
$$ \sigma_{n,0}, \sigma_{n,1}, \dots, \sigma_{n,k},\dots $$
so that
$$ \gamma_m(i)+\sigma_{n,0}<_m \gamma_m(i)+\sigma_{n,1}<_m\dots<_m \gamma_m(i)+\sigma_{n,k}<_m\dots.$$
Define a new sequence $(T_n)$ of trees as
$$ T_n=(S_{(n)_0})_{\sigma_{(n)_0, (n)_1}}. $$
Finally, define
$$ \tilde{\psi}_m((S_n)_{n\in\omega})(i)=\tilde{\varphi}((T_n)_{n\in\omega}). $$

\begin{lemma}\label{seqrankbelowomega2} Let $\mathcal{S}\subseteq \mathcal{T}^\omega$. Suppose for any coset tree $T\in \Psi(\mathcal{S})$, $\het(T)<\omega\cdot 2$. Then the following are equivalent for any $(S_n), (S'_n)\in\mathcal{S}$:
\begin{enumerate}
\item[\rm (i)] $(S_n)\tilde{E}_G (S'_n)$.
\item[\rm (ii)] For all $m\in\mathbb{N}_+$ there exist $i, i'\in\omega$ such that $\tilde{\psi}_m((S_n))(i)E_0^\omega \tilde{\psi}_m((S'_n))(i')$.
\item[\rm (iii)] For all $m\in\mathbb{N}_+$, both
\begin{itemize}
\item[\rm (a)] for any $i\in\omega$ there is $i'\in\omega$ such that $\tilde{\psi}_m((S_n))(i)E_0^\omega\tilde{\psi}_m((S'_n))(i')$ and
\item[\rm (b)] for any $i'\in\omega$ there is $i\in\omega$ such that  $\tilde{\psi}_m((S_n))(i)E_0^\omega\tilde{\psi}_m((S'_n))(i')$.
\end{itemize}
\end{enumerate}
\end{lemma}

\begin{proof}
The proof is identical to the proof of Lemma~\ref{rankbelowomega2}. The directions (i)$\Rightarrow$(iii)$\Rightarrow$(ii) do not require the boundedness assumption on the rank of coset trees. The boundedness assumption is only needed in the direction (ii)$\Rightarrow$(i).
\end{proof}

\begin{theorem}\label{oneplus} Let $G$ be a non-archimedean abelian Polish group and $(G_n)$ a decreasing sequence of open subgroups of $G$ which is a nbhd base of the identity element of $G$. Suppose $G_0=G$ and for all $n\in\omega$ and all prime $p$, $G_n/G_{n+1}$ is torsion and contains only finitely many elements of order $p$. Then any $G$-orbit equivalence relation is Borel reducible to $(E_0^\omega)^+$.
\end{theorem}

\begin{proof}
This follows from the proof of Lemma~\ref{anybase}, Theorem~\ref{allpcompact}, Lemma~\ref{seqrankbelowomega2}, and Lemma~\ref{plusomega}.
\end{proof}

We will show in the next section that $(E_0^\omega)^+$ in the above theorem as an upper bound for the orbit equivalence relation in the Borel reducibility hierarchy is not sharp. In fact, for any non-archimedean abelian Polish group $G$, no $G$-orbit equivalence relation can be Borel bireducible with $(E_0^\omega)^+$. 

We have developed two useful techniques so far in the discussion of orbit equivalence relations, namely, how to code sequences of trees (Lemma~\ref{seqfiniterank} and Lemma~\ref{seqrankbelowomega2}) and how to extend the coding map past a limit rank of possible coset trees (Lemma~\ref{rankbelowomega2}). The first does not increase the complexity of the equivalence relation used to reduce the orbit equivalence relation, and the second requires a jump operation $E\mapsto E^+$. By repeatedly applying these techniques to coset trees of higher rank, we get generalized versions of Lemma~\ref{seqrankbelowomega2} for the cases in which the heights of coset trees are strictly below $\omega\cdot 3$ and $\omega\cdot 4$. These generalized versions are nothing but a notational morass, and involve no new ideas. We obtain the following results.

\begin{theorem} \label{twopluses}Let $G$ be a non-archimedean abelian Polish group and $(G_n)$ a decreasing sequence of open subgroups of $G$ which is a nbhd base of the identity element of $G$. Suppose $G_0=G$ and for all $n\in\omega$, $G_n/G_{n+1}$ is torsion. Then any $G$-orbit equivalence relation is Borel reducible to $(E_0^\omega)^{++}$.
\end{theorem}

\begin{proof}
This follows from the proof of Lemma~\ref{anybase}, Corollary~\ref{alltorsion}, and a generalized version of Lemma~\ref{seqrankbelowomega2} for the case in which the heights of the coset trees are strictly below $\omega\cdot 3$.
\end{proof}

\begin{theorem} \label{threepluses} Let $G$ be a tame non-archimedean abelian Polish group. Then any $G$-orbit equivalence relation is Borel reducible to $(E_0^\omega)^{+++}$.
\end{theorem}

\begin{proof}
This follows from the proof of Lemma~\ref{anybase}, Corollary~\ref{generalrank}, and a generalized version of Lemma~\ref{seqrankbelowomega2} for the case in which the heights of the coset trees are strictly below $\omega\cdot 4$.
\end{proof}

By a theorem of Hjorth, Kechris and Louveau (\cite{HKL}, Theorem 2) on the potential Borel hierarchy of Borel equivalence relations, Theorem~\ref{threepluses} implies that any $G$-orbit equivalence relation for a non-archimedean abelian Polish group $G$ is {\em potentially} ${\bf\Pi}^0_6$, i.e., there exists a Polish topology of the underlying space for which the equivalence relation is ${\bf\Pi}^0_6$. This is in contrast with the results by Hjorth in \cite{Hj}, where he constructed examples of non-archimedean Polish groups which are TSI and tame, but have orbit equivalence relations of arbitrarily high potential Borel class.

\section{Essentially countable orbit equivalence relations}

In this section we consider orbit equivalence relations of non-archimedean abelian Polish groups that are essentially countable. In the proof of the following theorem we encounter the concept of direct sums of equivalence relations. For completeness we give the definition here. Let $(X_n)$ be a sequence of Polish spaces. The {\em direct sum space} $\bigoplus_n X_n$ is the Polish space whose underlying set is the disjoint union of all $X_n$, $n\in\omega$, and whose topology is given by
$$ U\subseteq \bigoplus_n X_n \mbox{ is open} \iff \mbox{for all $n\in\omega$, $U\cap X_n$ is open in $X_n$.} $$
Now if for each $n\in\omega$, $E_n$ is an equivalence relation on $X_n$, then the {\em direct sum} $\bigoplus_n E_n$ is an equivalence relation on $\bigoplus_n X_n$ defined as
$$ (x, y)\in \bigoplus_n E_n\iff \exists n\in\omega\ (x, y)\in E_n. $$
It is immediate that if each $E_n$ is hyperfinite, i.e., an increasing union of finite Borel equivalence relations, then so is $\bigoplus_n E_n$.

\begin{theorem}\label{essentialcountable}Let $G$ be a non-archimedean abelian Polish group. Let $E$ be an essentially countable Borel equivalence relation which is Borel reducible to a $G$-orbit equivalence relation. Then $E$ is essentially hyperfinite, i.e., $E\leq_B E_0$. 
\end{theorem}

\begin{proof}
Let $G$ be a closed subgroup of $H=\prod_n H_n$, where each $H_n$ is countable discrete abelian. By a theorem of Hjorth and Mackey (c.f., e.g., \cite{GaoBook} Theorem 3.5.2) the universal $G$-orbit equivalence relation is Borel reducible to the universal $H$-orbit equivalence relation. Thus $E$ is an essentially countable Borel equivalence relation which is Borel reducible to an $H$-orbit equivalence relation.

We now use a theorem of Hjorth and Kechris from \cite{HK}, Theorem 7.3, to get a sequence of Polish spaces $(Z_m)_{m\in\mathbb{N}_+}$, with a Borel action of $H^m$ on $Z_m$ for each $m\in\mathbb{N}_+$, such that 
$$ E\leq_B\bigoplus_{m\in\mathbb{N}_+} E^{Z_m}_{H^m}. $$

Since each $H^m$ is a countable discrete abelian group, each $E^{Z_m}_{H^m}$ is hyperfinite by the theorem of Gao and Jackson \cite{GJ}. It follows that $\bigoplus_m E^{Z_m}_{H^m}$ is hyperfinite. Thus $E$ is essentially hyperfinite.
\end{proof}

\begin{corollary} Let $G$ be a locally compact non-archimedean abelian Polish group. Then any $G$-orbit equivalence relation is essentially hyperfinite.
\end{corollary}

\begin{proof} By a theorem of Kechris (c.f. \cite{GaoBook} Theorem 7.5.2), any locally compact Polish group action induces only essentially countable Borel equivalence relations.
\end{proof}

\begin{corollary}\label{7thdichotomy} Let $G$ be a non-archimedean abelian Polish group, $F$ be a Borel $G$-orbit equivalence relation and $E\leq_B F$. Then exactly one of the following holds:
\begin{enumerate}
\item[\rm (I)] $E$ is essentially hyperfinite, i.e., $E\leq_B E_0$. 
\item[\rm (II)] $E_0^\omega\leq_B E$.
\end{enumerate}
\end{corollary}

\begin{proof} This follows immediately from the Seventh Dichotomy Theorem (\cite{HK} Theorem 8.1) of Hjorth and Kechris, which is stated for TSI groups with option (I) being $E$ essentially countable.
\end{proof}

Now we can see that the upper bounds in Theorems ~\ref{oneplus}, \ref{twopluses} and \ref{threepluses} are not sharp. In fact, let $G$ be a tame non-archimedean abelian Polish group. Let $E$ be any countable Borel equivalence relation that is not hyperfinite (c.f. \cite{GaoBook} Theorem 7.4.10 for an example). By the Seventh Dichotomy Theorem, $E_0^\omega\not\leq_B E$. It follows from Corollary~\ref{7thdichotomy} that $E$ is not Borel reducible to any Borel $G$-orbit equivalence relation. On the other hand, let $\mbox{id}(2^\omega)$ be the identity relation on the Polish space $2^\omega$. Then $E\leq_B \mbox{id}(2^\omega)^+$  (c.f. \cite{GaoBook} Exercise 8.3.3) and $\mbox{id}(2^\omega)^+\leq_B (E_0^\omega)^+$ since $\mbox{id}(2^\omega)\leq_B E_0^\omega$ (c.f. \cite{GaoBook} Theorem 5.3.5 and Exercise 8.3.2). Thus $E\leq_B (E_0^\omega)^+$. Therefore, $(E_0^\omega)^+$ is not Borel reducible to any $G$-orbit equivalence relation.

\section{Universal tame product groups}

Let $\mathcal{P}$ denote the class of all tame groups of the form $\prod_n H_n$, where each $H_n$ is countable discrete abelian. In this section we prove that $\mathcal{P}$ has a universal element, i.e., there exists $H\in \mathcal{P}$ such that any $G\in \mathcal{P}$ is a closed subgroup of $H$.

We will use the following trivial facts tacitly.

\begin{lemma} Let $H=\prod_n H_n$, where each $H_n$ is a countable discrete abelian group. Then the following hold:
\begin{enumerate}
\item[\rm (a)] If $G=\prod_n G_n$, where each $G_n$ is a subgroup of $H_n$, then $G$ is a closed subgroup of $H$.
\item[\rm (b)] For any strictly increasing sequence $(n_k)_{k\in\omega}$ of natural numbers with $n_0=0$, if we define
$$ L_k=\prod_{n_k\leq i<n_{k+1}}H_i $$
for $k\in\omega$, then $H$ is isomorphic to $\prod_k L_k$.
\item[\rm (c)] More generally, if $\mathcal{F}=\{F_0, F_1, \dots, F_k,\dots\}$ is a partition of $\omega$ into finite subsets, letting
$$ L_k=\prod_{i\in F_k} H_i $$
for $k\in\omega$, then $H$ is isomorphic to $\prod_k L_k$.
\end{enumerate}
\end{lemma}

Fix an enumeration of all primes $p_0, p_1,\dots, p_n,\dots$. We now define a universal tame product group $H_\infty=\prod_n H_n$ as follows. Let 
$$ H_0=A_\infty=\bigoplus_{p\in\mathbb{P}}\mathbb{Z}(p^\infty)^{<\omega}\oplus\mathbb{Q}^{<\omega}. $$
For $n\in\omega$ let 
$$ H_{n+1}=\bigoplus_{0\leq i\leq n} \mathbb{Z}(p_i^\infty)\oplus\bigoplus_{i>n}\mathbb{Z}(p_i^\infty)^{<\omega}. $$

\begin{proposition}\label{universalproduct} $H_\infty$ is a universal group in $\mathcal{P}$.
\end{proposition}

\begin{proof}
To see that $H_\infty$ is tame, we only need to use Theorem~\ref{soleckimain} and Proposition~\ref{soleckipcompact}, and note that for any $i\in\omega$, $H_n$ is $p_i$-compact when $n>i$.

Before proving the universality of $H_\infty$, we define an isomorphic copy of $H_\infty$.   Let $A_0=\{0\}$. For each $n\in\omega$, let 
$$ A_{n+1}=\bigoplus_{i>n}\mathbb{Z}(p_i^\infty)^{<\omega}. $$
Then for each $n\in\omega$, $H_n\times A_n$ is isomorphic to $H_n$, and therefore $H_\infty$ is isomorphic to $\prod_n (H_n\times A_n)$. It suffices to show the universality of the latter group.

Now let $G\in\mathcal{P}$ and suppose $G=\prod_n G_n$, where each $G_n$ is countable discrete abelian. Since every countable $p$-group is a subgroup of a countable divisible $p$-group, and each countable divisible $p$-group is a direct sum of up to countably infinitely many copies of $\mathbb{Z}(p^\infty)$ (c.f. \cite{Ro} 4.1.5 and 4.1.6), we may assume without loss of generality that the $p$-component of $G_n$ is of the form $\mathbb{Z}(p^\infty)^m$ for some $m\in\omega$ or of the form $\mathbb{Z}(p^\infty)^{<\omega}$.

Inductively define a sequence $(n_k)$ as follows. Let $n_0$ be the least such that for all $n\geq n_0$, $G_n$ is $p_0$-compact. In particular, for all $n\geq n_0$, $G_n$ is torsion, and hence by our assumption is a direct sum of its $p$-components for all $p\in\mathbb{P}$. If $n_k$ is already defined, let $n_{k+1}>n_k$ be the least such that for all $n\geq n_{k+1}$, $G_n$ is $p_{k+1}$-compact.  

Define $L_0=\prod_{i<n_0} G_i$, if $n_0>0$; otherwise let $L_0=\{0\}$. For $k>0$, let
$$ L_k=\prod_{n_{k-1}\leq i<n_k} G_i. $$
Then $\prod_n L_n$ is isomorphic to $\prod_n G_n$, and the $p$-component of each $L_n$ is still of the form $\mathbb{Z}(p^\infty)^m$ for some $m\in\omega$ or of the form $\mathbb{Z}(p^\infty)^{<\omega}$. Now for all $n>k$, $L_n$ is $p_k$-compact, and therefore its $p_k$-component is of the form $\mathbb{Z}(p_k^\infty)^m$ for some $m\in\omega$. Denote this $m$ by $m(n, k)$.

For all $n\in\mathbb{N}_+$, define $$M_n=\max\{ m(n,k)\,:\, k<n\} $$
and inductively define
$$ N_1=1, \ \ N_{n+1}=N_n+M_n+1. $$
Then $(N_n)$ is strictly increasing, and for all $n\in\mathbb{N}_+$, $N_n\geq n$.

Define $K_0=H_0=H_0\times A_0$. For $n\in\mathbb{N}_+$, define
$$ K_n=A_n\times\prod_{N_n\leq i<N_{n+1}} H_i. $$
Then $\prod_n K_n$ is isomorphic to $\prod_n (H_n\times A_n)$, and it suffices to show that each $L_n$ is a subgroup of $K_n$. 

For $n=0$, $K_0=H_0=A_\infty$ is a universal countable abelian group, and thus $L_0\leq K_0$. Now fix $n\in\mathbb{N}_+$. By our construction, for every $k<n$, $L_n$ is $p_k$-compact, and its $p_k$-component is of the form $\mathbb{Z}(p_k^\infty)^{m(n,k)}$, while the $p_k$-component of $K_n$ is of the form $\mathbb{Z}(p_k^\infty)^{M_n+1}$. For $k\geq n$, the $p_k$-component of $K_n$ is of the form $\mathbb{Z}(p_k^\infty)^{<\omega}$. Thus for all $k\in\omega$, the $p_k$-component of $L_n$ is a subgroup of the $p_k$-component of $K_n$. Since $L_n$ and $K_n$ are both torsion, we have that $L_n$ is a subgroup of $K_n$.
\end{proof}

\section{Other remarks and open problems}

There are many questions left open on the subject of this paper. On the structure of tame non-archimedean abelian Polish groups, the following is a basic question with the most potential impact.

\begin{question} Is every tame non-archimedean abelian Polish group a closed subgroup of a tame group which is a countable product of countable discrete abelian groups?
\end{question}

A positive answer would not only simplify many results of this paper but also imply the existence of a universal tame non-archimedean abelian Polish groups by Proposition~\ref{universalproduct}. 

In Solecki's Theorem~\ref{soleckimain}, the necessity direction is true without the abelian assumption. Thus in particular it applies to all countable products of countable discrete groups. We do not know if our Lemma~\ref{nt2np}, which is a generalization of Solecki's necessity condition, is still true for TSI groups, i.e., closed subgroups of countable products of countable discrete groups. Our current proof of Lemma~\ref{nt2np} for the abelian case uses Lemma~\ref{extansion}, which is demonstrably false for non-abelian groups.

Our interest on the existence of a universal, surjectively universal, or weakly universal object in a class of Polish groups is because these questions are good test questions on our understanding on the structure of the groups. Unfortunately all of these questions are open for tame non-archimedean abelian Polish groups. Chang \cite{Ch} has shown that there is a surjectively universal locally compact non-archimedean abelian Polish group. 

In view of Proposition~\ref{universalproduct}, we are particularly interested in the question on universality.

\begin{question} Is there a universal tame non-archimedean abelian Polish group?
\end{question}

Turning to the orbit equivalence relations of tame non-archimedean abelian Polish groups, we already remarked before that none of the upper bounds in the form of $E^+$ are sharp. For the rank of group trees, we know that the bound $\omega$ in Theorem~\ref{allpcompact} and the bound $\omega\cdot 2$ in Corollary~\ref{alltorsion} are sharp, but we do not know that the bound $\omega\cdot 3$ in Theorem~\ref{generalrank} is.

We also do not know that the bound on the potential Borel class of the orbit equivalence relations, that they are all potentially ${\bf\Pi}^0_6$, is sharp.  In fact, we propose the following bold conjecture.

\begin{conjecture} Let $G$ be any non-archimedean abelian Polish group. Then every $G$-orbit equivalence relation is Borel reducible to $E_0^\omega$, and therefore is potentially ${\bf\Pi}^0_3$.
\end{conjecture}

Finally, Theorem~\ref{essentialcountable} is a special case of the following general conjecture.

\begin{conjecture} Let $G$ be any abelian Polish group and $F$ an orbit equivalence relation induced by a Borel action of $G$ on a Polish space. If $E\leq_B F$ and $E$ is essentially countable, then $E$ is essentially hyperfinite.
\end{conjecture}

This conjecture was first mentioned in \cite{Hj} as a passing remark, and has gained traction since the settlement of the countable case \cite{GJ}. Our Theorem~\ref{essentialcountable} provides more evidence, beyond which the conjecture is wide open.

\end{document}